\documentclass[a4paper]{amsart}
\usepackage{a4}
\usepackage{amsmath}
\usepackage{amssymb}
\usepackage{enumerate}
\usepackage[applemac]{inputenc}
\usepackage{graphicx}
\usepackage{xcolor}
\usepackage{multirow}
\usepackage{subcaption}
\usepackage[ruled, lined, linesnumbered, commentsnumbered, longend]{algorithm2e}
\usepackage{xcolor}

\newcommand{\PP}{ \mathbb{P}}

\newcommand{\M}{{\mathbb M}}

\newcommand{\EE}{{\mathbb E}}

\newcommand{\ds}{\displaystyle}
\newcommand{\supp}{\mathrm{supp}\;}

\newcommand{\diver}{{\rm{div}}}

\DeclareMathOperator{\divergence}{div}

\PassOptionsToPackage{normalem}{ulem}
\usepackage{ulem}

\newtheorem{remark}{\textbf{Remark}}[section]

\newtheorem{lemma}{\textbf{Lemma}}[section]
\newtheorem{theorem}{\textbf{Theorem}}[section]

\newtheorem{proposition}{\textbf{Proposition}}[section]

\numberwithin{equation}{section}

\title[High-order scheme for MFGs system]{A high-order scheme for mean field games} 

\author{Elisa Calzola \and Elisabetta Carlini \and Francisco J. Silva}

\thanks{``Sapienza'', Universit\`a di Roma, Dipartimento di Matematica Guido Castelnuovo, 00185 Rome, Italy (calzola@mat.uniroma1.it, carlini@mat.uniroma1.it)}

\thanks{Institut de recherche XLIM-DMI, UMR 7252 CNRS, Facult\'e des Sciences et Techniques,
	Universit\'e de Limoges, 87060 Limoges, France (francisco.silva@unilim.fr)}

\def\dd{{\rm d}}

\newcommand{\ov}[1]{\overline{#1}}

\def\weight(#1,#2){c_{#1,#2}}

 \if {
  
 } \fi

\if {

}{{\caption}}
} \fi

\def\Dt{\Delta t}
\def\Dx{\Delta x}

\def\D{\mathcal{D}}

\def\F{\mathcal{F}}

\def\I{\mathcal{I}}

\def\M{\mathcal{M}}

\def\P{\mathcal{P}}

\def\eps{\varepsilon}

\def\supp{\mathop{\rm supp}}

\def\half{\mbox{$\frac{1}{2}$}}

\def\1B{{\bf  1}}

\newcommand{\NN}{\mathbb{N}}
\newcommand{\ZZ}{\mathbb{Z}}

\newcommand{\OO}{\mathcal{O}}
\newcommand{\RR}{\mathbb{R}}

\def\EE{\mathbb{E}}
\def\PP{\mathbb{P}}

\newcommand\be{\begin{equation}}
\newcommand\ee{\end{equation}}
\newcommand\ba{\begin{array}}
\newcommand\ea{\end{array}}
\newcommand{\bean}{\begin{eqnarray*}}
\newcommand{\eean}{\end{eqnarray*}}

\def\ds{\displaystyle}

\begin{document}

\begin{abstract}
In this paper we propose a high-order numerical scheme for time-dependent mean field games systems. The scheme, which is built by combining Lagrange-Galerkin and semi-Lagrangian techniques, is consistent and stable for large time steps compared with the space steps. We provide a convergence analysis for the exactly integrated Lagrange-Galerkin scheme applied to the Fokker-Planck equation, and we propose an implementable version with inexact integration. Finally, we validate the convergence rate
of the proposed scheme through the numerical approximation of two mean field games systems.
\end{abstract}

\maketitle

{\small
}

\section{Introduction} 
\label{Sect_Intro} 

This work concerns the numerical approximation of Mean Field Games (MFGs), introduced simultaneously by Lasry-Lions in \cite{LasryLions06i,LasryLions06ii,LasryLions07} and by Huang-Caines-Malham\'e in
 \cite{HMC06}. MFGs characterize Nash equilibria of stochastic differential games with an infinite number of indistinguishable players. In some specific instances, the aforementioned equilibria are described by a system of parabolic Partial Differential Equations (PDEs) consisting of a Hamilton-Jacobi-Bellman (HJB) equation, with a terminal condition, coupled with a Fokker-Planck (FP) equation with an initial condition.


The numerical approximation of MFGs has been an active area of research over the last decade (see e.g.~\cite{MR4214777,MR4368188} and the references therein). Let us mention, for instance, the articles~\cite{AchdouCapuzzo10} and~\cite{CS15} proposing a semi-implicit finite difference scheme and a Semi-Lagrangian (SL) type scheme, respectively. The scheme studied in \cite{CS15}, which allows for large time steps compared to space steps, has been extended in \cite{MR3828859} to deal with nonlinear FP equations and in \cite{Jakobsen_et_al_2021} to approximate MFGs with non-local diffusions. On the other hand, to the best of our knowledge, only few works deal with high-order numerical schemes for MFG systems. Let us mention \cite{PT15} and \cite{MR4253925}, where the authors propose finite difference based second-order accurate methods, and the recent contribution~\cite{MR4622007}, where high-order space-time finite elements are used to approximate variational MFGs.  

The main purpose of this article is to provide a new high-order approximation scheme, meaning an order of convergence larger than two, for a class of second-order MFG systems with constant diffusion. The scheme combines a high-order Lagrange-Galerkin (LG) discretization for the FP equation with a high-order SL discretization for the HJB equation.  The main novelty of our scheme lies in the discretization of the FP equation which, inspired by \cite{Morton88} and \cite{MR3828859}, is constructed by using SL techniques for the time discretization (see e.g. \cite{CamFal95,falconeferretilibro}) and LG techniques for the space discretization (see e.g. \cite{Morton88,Bermejo12}). More precisely, the stochastic characteristic curves of the FP equation are approximated with a Crank-Nicolson method (see e.g. \cite{MR1214374,milstein:2013}), as in high-order SL schemes for parabolic equations (see \cite{BCCF21}), and the space variable is discretized by using a LG scheme with a symmetric Lagrangian basis of odd order. This last choice is inspired by the results in \cite{Ferretti2013, Ferretti20}, where the equivalence between SL and LG schemes has been studied, and where symmetric odd basis have shown a good behavior in terms of stability. The resulting scheme for the FP equation is explicit, conservative, consistent, stable, allows for large time steps compared with space steps, is convergent, and high-order accurate. When coupled with a high-order SL for the HJB equation, one obtains a high-order scheme for the MFG system which, because of its forward-backward structure, is not explicit and is solved by fixed-point iterations. We numerically show high-order accuracy of the scheme by considering two MFG systems. The first one is a linear-quadratic MFG with non-local couplings (see e.g. \cite{MR3489817}), for which we are able to compute its analytical solution, and the second one, taken from \cite{PT15}, is a MFG with local couplings (see e.g. \cite{MR4214774}) and no explicit solution.

The article is organized as follows. In Section \ref{sec_mfg}, we recall the MFG system we are interested in, as well as some basic results on FP equations. Section \ref{lagrange_galerkin_first_order_case} introduces a new scheme for FP equations, based on SL techniques and LG approximations, and establish its main properties.
In Section~\ref{sec_schemeMFG}, we present a high-order SL scheme for HJB equations and couple it with the scheme for the FP equation studied in Section~\ref{lagrange_galerkin_first_order_case} to derive a new scheme for the MFG system. Finally, in Section~\ref{sec:numerics} we provide an implementable version of the method, derived from the use of a cubic basis and Simpson's rule in the LG approximation. The paper concludes by showing the performance of the proposed scheme in two examples: a linear-quadratic MFG with non-local couplings and admitting an explicit solution, and a MFG with local couplings and without explicit solutions. In all the numerical examples, an order of accuracy between two and three is observed. Finally, we provide in the Appendix of this work the proof of some needed technical results.

\section{Preliminary results}\label{sec_mfg}
In the following, given a function $u: [0, T]\times \RR^d\to \RR$ and $(t,x)\in (0,T)\times \RR^d$,  the notations $\nabla u(t,x)$ and $\Delta u(t,x)$ refer to the gradient and Laplacian of $u$ with respect to the spatial variable $x$. Similarly, given $v: [0, T]\times \RR^d\to \RR^d$, the notation $Dv$ and $\text{div}(v)$ refer  to the Jacobian matrix and the divergence of $v$ with respect to the space variable, respectively.
We also denote by $(\mathcal{P}_1(\RR^d),{\bf{d}})$ be the metric space of Borel probability measures on $\RR^d$ with finite first order moment, endowed with the $1$-Wasserstein distance ${\bf{d}}$  (see e.g. \cite[Section 7.1]{Ambrosiogiglisav} for the definition of ${\bf{d}}$).

We focus on the numerical approximation of the following time-dependent second-order MFG  with non-local couplings (see~\cite{LasryLions06ii,LasryLions07}): 
\be\ba{rcl}\label{MFG} \tag{{\bf MFG}}
-\partial_{t} v -\frac{\sigma^2}{2} \Delta v+H(x,\nabla v)   &=& F(x, m(t)) \;  \;    \hbox{in }   [0,T)\times\RR^d, \\[6pt]
\partial_{t} m  -\frac{\sigma^2}{2} \Delta m-\diver\big( \partial_{p}H(x,\nabla  v) m\big) &=&0 \; \; \; \hbox{in }(0,T]\times  \RR^d, \\[6pt]
v(T,\cdot)= G(\cdot, m(T)),   & \; & \; m(0,\cdot)= m^*_0 \quad \mbox{in }  \RR^{d},
\ea\ee
\normalsize
where $T>0$, $\sigma\in \RR \setminus \{0\}$, $\RR^d\times \RR^d \ni (x,p) \mapsto  H(x,p) \in \RR$ is convex and differentiable with respect to $p$, $F$, $G: \RR^d\times \P_1(\RR^d)\to \RR$, and $m^*_0: \RR^d\to \RR$. Notice that \eqref{MFG} consists of a HJB equation, with a terminal condition, coupled with a FP equation with an initial condition. 
For the sake of simplicity,  in what follows we will suppose that the {\it Hamiltonian} $H$ is purely quadratic, i.e. $H(x,p)= |p|^2/2$ for all $x,\, p \in \RR^d$ and we assume that: \smallskip\\
{\bf(H1)}  $m_0^*$ is nonnegative, H\"older continuous, has compact support, and $\int_{\RR^d}m_0^*(x) \dd x=1$. \smallskip
. \smallskip\\
{\bf(H2)}  $F$ and $G$ are bounded and Lipschitz continuous. Moreover, for every $\mu \in \P_1(\RR^d)$, $F(\cdot, \mu)$ is of class $C^2$  and 
$$
\sup_{x\in \RR^d, \mu \in \P_1(\RR^d)} \left\{ \|DF(x,\mu)\|_{\infty} +\|D^2F(x,\mu)\|_{\infty}\right\} <\infty.
$$

Under {\bf(H1)-\bf(H2)}   system \eqref{MFG} admits at least one classical solution $(v^*,m^*)$ (see e.g. \cite[Theorem 3.1]{Cardialaguet10}). Moreover, if the coupling terms $F$ and $G$ satisfy a monotonicity condition with respect to $m$, then the classical solution is unique (see \cite[Theorem 2.4]{LasryLions07}). 

In  order to obtain a high-order scheme for \eqref{MFG}, our first task will be to construct a high-order LG scheme for the following linear FP equation:
\be
\label{FP}
\ba{rcl}
 \partial_t m -\frac{\sigma^2}{2} \Delta m +\mbox{div}\left(b m  \right) &=& 0 \quad \mbox{in } (0,T)\times \RR^d, \\[6pt]
 m(0,\cdot)&= &  m^{\ast}_{0} \quad \mbox{in } \RR^d,
\ea \tag{{\bf FP}}
\ee
where $\sigma \in \RR\setminus\{0\}$,  $ b: [0, T]\times \RR^d \to \RR^d$, and  $m_0^*:\RR^d \to \RR$.  
We will assume that:\smallskip\\
{\bf(H3)}  \smallskip
  $b\in C([0,T]\times \RR^d)$,  $b$ is bounded
and there exists $C_{b}>0$ such that 
  $$| b(t,x)-b(t,y)| \leq C_{b}|x-y|, \quad \text{for $t\in [0,T]$  and $x,\, y\in \RR^d$}.
 $$  
\smallskip
In the following result, proved in the Appendix, we summarize some properties of equation \eqref{FP}. 
\begin{theorem}
\label{teorema_well_posedness} 
Assume {\bf(H1)} and {\bf(H3)}.  Then the following hold: 
\begin{enumerate}
\item[{\rm(i)}]  Equation~\eqref{FP} admits a unique classical solution $m^*\in C^{1,2}([0,T]\times \RR^d)$.  
\item[{\rm(ii)}] $m^* \geq 0$. 
\item[{\rm(iii)}]   $\int_{\RR^d} m^*(t,x) \dd x=1$ for all $t\in[0,T]$.
\item[{\rm(iv)}]  $m^*$ is the unique solution in $L^2([0,T]\times \RR^d)$ to \eqref{FP}  in the distributional sense. 
\end{enumerate}
\end{theorem}

Let us recall the probabilistic interpretation of the solution $m^*$ to \eqref{FP}, which will be useful in order to construct a LG scheme. Let $W$ be a $d$-dimensional Brownian motion defined on a probability space $(\Omega, \F, \PP)$ and let $Y_0:\Omega \to \RR^d$ be a random variable, independent of $W$, and whose distribution is absolutely continuous with respect to the Lebesgue measure in $\RR^d$, with density given by $m_0^*$. Given $(t,x) \in [0,T]\times \RR^d$, we define $Y^{t,x}$ as the unique strong solution to the SDE:
\be
\label{SDE_underlying_FP}
\ba{rcl} \dd Y(s) &=& b(s,Y(s)) \dd s  + \sigma \dd W(s) \quad \text{for $s\in (t,T)$},\\[6pt]
		      Y(t) &=& x. 
\ea
\ee
Denote by $\EE(X)$ the expectation of a random variable $X:\Omega\to \RR$. Under the asspumptions of Theorem \ref{teorema_well_posedness},  $Y^{0,Y_0}(t)$ is well defined for all $t\in [0,T]$ and its distribution is absolutely continuous with respect to the Lebesgue measure in $\RR^d$, with density given by $m^*(t,\cdot)$ (see e.g. \cite{figalli08}).  From the $\PP$-a.s. equality $Y^{0,Y_0}(s)= Y^{t,Y^{0,Y_0}(t)}(s)$ for every $0\leq t\leq s \leq T$, we deduce that for every  continuous  and bounded function $\phi: \RR^d\to \RR$, we have 
\be
\label{basic_formula_for_the_scheme}
\int_{\RR^d} \phi(x) m^*(s,x)\dd x=\int_{\RR^d}  \EE \left( \phi( Y^{t,x}(s)) \right)m^*(t,x)\dd x.
\ee


\section{A Lagrange-Galerkin type scheme for a  Fokker-Planck equation} 
\label{lagrange_galerkin_first_order_case}

Let us focus on the numerical approximation of $\eqref{FP}$. Notice that if $b$ is differentiable with respect to the space variable,   \eqref{FP} can be written as  
$$
\begin{aligned}
	\partial_t m -\frac{\sigma^2}{2} \Delta m + \langle b,\nabla m\rangle +\mbox{{\rm div}}(b)m &= 0 &\mbox{in } (0,T)\times \RR^d, \\
	m(0,\cdot)&= m_{0}^* & \mbox{in } \RR^d.
\end{aligned} 
$$
Using this formulation, a second-order accurate semi-Lagrangian scheme can be derived to approximate $m^*$ (see e.g. \cite{BCCF21}). However, such a scheme is not conservative, i.e. the discrete solution does not satisfy the discrete analogous of Theorem~\ref{teorema_well_posedness}{\rm(iii)}. The scheme that we consider, which will be built from \eqref{basic_formula_for_the_scheme}, will allow us to preserve this property (see Theorem~\ref{Prop:stability}{\rm(ii)} below).

Let us fix $N_{\Delta t} \in \NN$, set $\I_{\Dt}=\{0,\hdots, N_{\Delta t}\}$, $\I_{\Dt}^*=\I_{\Dt}\setminus \{N_{\Delta t}\}$, $\Delta t = T/N_{\Delta t}$, and $t_{k}= k\Delta t$ ($k\in \I_{\Dt}$). Let $x\in\RR^{d}$ and consider the sequence of random variables $(y_{k})_{k=0}^{N_{\Delta t}}$ defined by $y_0=x$ and, for every $k\in\I_{\Delta t}^{*}$, $y_{k+1}$ is the unique solution to
\be\label{eq:CN}
y=y_{k}+\frac{\Dt}{2}\left(b(t_k,y_{k})+b(t_{k+1},y)\right) +\sqrt{\Dt}\sigma\xi_k,
\ee
where $(\xi_{k})_{k=0}^{N_{\Delta t}-1}$ is a sequence of i.i.d. $\RR^d$-valued random variables with i.i.d. components such that, for every $k\in\I_{\Delta t}^{*}$,
\be
\label{rv}
\PP((\xi_k)_i=0)=2/3 \quad \text{and}\quad \PP((\xi_k)_i=\pm \sqrt{3})=1/6\quad \mbox{for all }\quad i=1,\hdots,d.
\ee

Since $b$ is Lipschitz continuous, the sequence $(y_{k})_{k=0}^{N_{\Delta t}}$, called the Crank-Nicolson (CN) approximation of $Y^{0,x}$, is well-defined for $\Delta t$ sufficiently small. 

An interesting feature of the law of $(\xi_{k})_i$ in~\eqref{rv} is that, provided that $b$ is smooth enough, $(y_{k})_{k=0}^{N_{\Delta t}}$ is a second order weak approximation of $Y^{0,x}$ (see e.g.~\cite[Section 15.4, equation (4.11)]{MR1214374} and also \cite[Section 2, Table 1]{ferrettisecondoordine}), i.e.
for every $\phi:\RR^d\to \RR$ smooth enough and for every $k\in\I_{\Delta t}^{*}$, we have
 \be \label{stimaCN}
  \big|\EE\left(\phi(y_k)\right) - \EE \left( \phi(Y^{0,x}(t_{k}))\right)\big|
= O\left( (\Dt)^2\right).
 \ee
Notice that this estimate is better than the one obtained by considering a classical random walk in $\RR^{d}$, i.e. when distribution of $(\xi_{k})_i$ is given by
\be
\label{rv0}
\PP((\xi_k)_i=\pm 1)=1/2\quad \mbox{for all }\quad i=1,\hdots,d,
\ee
for which it is known that second order accuracy does not hold (see Sections 5.1.A and 5.1.B in~\cite{KPS}).

In order to discretize~\eqref{basic_formula_for_the_scheme}, for every $k\in\I_{\Delta t}^{*}$ and $x\in \RR^d$, denote by $y^{t_k,x}$ the one-step CN approximation of $Y^{t_{k}, x}(t_{k+1})$, given by the unique solution to \eqref{eq:CN}.
Let $\I_{d}=\{1, \hdots, 3^d\}$, define $\{ e^{\ell} \, | \, \ell\in \I_{d}\} \subset \RR^d$ as the set of possible values of $\xi_{k}$, set $\omega^\ell=\PP(\xi_k=e^\ell)$, and denote by $y^\ell_k(x)$ the unique solution to \eqref{eq:CN} for $\xi_k=e^\ell$ ($\ell \in \I_d$). By setting $t=t_k$, $s=t_{k+1}$, and replacing $Y^{t_{k}, x}(t_{k+1})$ by $y^{t_k,x}$ in~\eqref{basic_formula_for_the_scheme}, we obtain  the following semi-discrete scheme for~\eqref{FP}:

 \be\label{semidiscreteintime}
\int_{\RR^d} \phi(x) m_{k+1}(x)\dd x=\sum_{\ell\in \I_{d}}\omega_\ell \int_{\RR^d}\phi( y^\ell_k(x) )m_k(x)\dd x \quad \text{for $\phi$ continuous and bounded, $k\in \I_{\Dt}$,}
 \ee
with $m_0= m_{0}^*$ and unknowns $\{m_{k}: \RR^d \to \RR \, | \, k\in \I_{\Dt}\setminus \{0\}\}$.
Note that the assumption that $m_0$ has a compact support implies the existence of $L_{\Delta t} = O(1/\sqrt{\Delta t})$ such that the solution $m_{\Delta t}$ to \eqref{semidiscreteintime} satisfies
\be\label{Compact-support}
\supp(m_{\Delta t,k}) \subset [-L_{\Delta t},L_{\Delta t}]^{d} \quad \text{for $k\in \I_{\Dt}$}.
\ee

In order to construct a space discretization of \eqref{semidiscreteintime}, and hence a fully-discrete scheme for~\eqref{FP}, we consider a symmetric Lagrangian basis of odd order. More precisely, let us fix $p\in \NN$, set $q:=2p+1$, and let $\widehat{\beta}: \RR \to \RR$ be defined by
\be\label{eq:referencebasis}
(\forall \, \xi \in [0,\infty)) \quad \widehat{\beta}(\xi)= \begin{cases} \ds \prod_{k\neq 0, \, k=-p}^{p+1} \frac{\xi-k}{-k} & \text{if } \xi \in [0,1],\\[16pt]
\ds \prod_{k\neq 0, \, k=-p+1}^{p+2} \frac{\xi-k}{-k} & \text{if } \xi \in (1,2],\\[6pt]
\vdots & \; \\[6pt]
\ds \prod_{k=1}^{2p+1} \frac{\xi-k}{-k} & \text{if } \xi \in (p,p+1],\\[12pt]
 0 & \text{if } \xi \in (p+1, \infty), \\[12pt]
 \widehat{\beta}(-\xi) & \text{if } \xi \in (-\infty, 0).
 \end{cases} 
\ee

Following \cite{Ferretti2013}, for $\Delta x \in (0,\infty)$, we consider the symmetric Lagrange interpolation basis functions $\{\beta_i\}_{i \in {\ZZ^d}}$ defined as
$$
(\forall \, z=(z_1,\dots z_d)\in \RR^d, \; i=(i_1,\dots,i_d)\in \ZZ^d) \quad  \beta_i(z)=\prod_{j=1}^{d} \widehat \beta\left(\frac{z_j}{\Dx}-i_j\right).
$$
  
For all $i\in \ZZ^d$, let us set $x_i = i \Delta x$.  Notice that  $\beta_i$ has compact support, $\beta_i(x_j)=1$ if $i=j$ and $ \beta_i(x_j)=0$ otherwise, and, for all $x\in \RR^d$, $\sum_{j \in \ZZ^d} \beta_j(x) =1$.
Given  $f \in W^{q+1,\infty}(\RR^d) $,  we define the interpolant $I[f]: \RR^d \to \RR$  by 
\be
\label{definizione_I}
I[f](x)= \sum_{i \in \ZZ^d}f(x_i)\beta_{i}(x) \quad \text{for } x\in \RR^d.
\ee 
By \cite[Theorem 16.1]{Ciarlet},  the following estimate holds
\be\label{approximation_error_V_Deltax_q}
\sup_{x\in \RR^d} |f(x) -I[f](x)|\leq C_I (\Delta x)^{q+1} \| D^{q+1} f \|_{L^\infty},
\ee
where $C_I>0$ is  independent of $f$ and $\Delta x$.  Notice that in the one dimensional case ($d=1$),   $I[f]$ restricted to a given interval $(x_i, x_{i+1})$  ($i\in \ZZ$) is the Lagrange interpolating polynomial of degree $q$ constructed on the symmetric stencil   $x_{i-(q-1)/2},\dots, x_{i+1+(q-1)/2}$.

Let $L_{\Delta t}>0$ be as in \eqref{Compact-support},  let $N_{\Delta x}\in \NN$, and set $\I_{\Dx}\:=\{-N_{\Dx},\dots,N_{\Dx}\}^d$. From now on, we assume that $\Delta x= L_{\Delta t}/N_{\Delta x}$, we set    
$$\Delta=(\Delta t, \Delta x), \quad \text{and} \quad \OO_{\Delta}=[-L_{\Dt} - p\Dx, L_{\Dt}+ p\Dx]^d.$$ 

We look for an approximation $m_{\Delta}$ of the solution $m^*$ to \eqref{FP} such that, for all $ k\in \I_{\Delta t}$,
\be\label{m_Delta}
m_{\Delta}(t_k,x) =\sum_{i \in  \I_{\Delta x}}m_{k,i}\beta_{i}(x)\;\text{for}\; x\in \OO_{\Delta },\quad
m_{\Delta}(t_k,x) =0\;\text{for}\; x\in \RR^d\setminus \OO_{\Delta },\ee
where  $m_{k,i}\in \RR$  ($k\in \I_{\Dt}$,  $i\in \I_{\Dx}$) have to be determined.  Notice that, by definition of $\I_{\Delta x}$, for all $k\in \I_{\Delta t}$ we have that $\mbox{supp}\{m_{\Delta}(t_k,\cdot)\}\subset \OO_{\Delta}$. 
 Replacing $m$ by $m_{\Delta}$ and taking $\phi=\beta_i$ ($i\in \I_{\Dx}$) in \eqref{semidiscreteintime} yields the  following explicit iterative  scheme for the unknowns $m_{k,i}\in \RR$  ($k\in \I_{\Dt}$,  $i\in \I_{\Dx}$) 
\be \label{LG}
\ba{rrl} 
\ds \sum_{j\in  \I_{\Delta x}} m_{k+1,j}\int_{\OO_{\Delta}  }\beta_{i}(x)\beta_j(x) \dd x&=&\ds \sum_{j\in  \I_{\Delta x}}  m_{k,j}\sum_{\ell\in \I_{d}}\omega_\ell \int_{\OO_{\Delta}  } \beta_i ( y^\ell_k(x)) \beta_j(x)\dd x \\[15pt]
\; & \; &  \hspace{3cm} \text{for $k\in \I^*_{\Dt}$, $i\in \I_{\Delta x}$}, \\[4pt]
\ds \sum_{j \in  \I_{\Delta x}} m_{0,j} \int_{\OO_{\Delta}} \beta_{i}(x) \beta_{j}(x) \dd x &=&\ds \int_{\OO_{\Delta}} m_0^*(x)\beta_{i}(x) \dd x. \hspace{0.6cm} 
\ea
\ee
\smallskip

Let  $A$  be the $(2N_{\Dx}+1)^d\times (2N_{\Dx}+1)^d$ real mass matrix with entries given by 
\be
\label{def_mass_matrix}
A_{i,j}=\int_{ \OO_{\Delta}   }\beta_i(x)\beta_j(x) \dd x \quad \mbox{for} \quad (i,j)\in \I_{\Dx}\times\I_{\Dx}.
\ee
For $k\in \I_{\Dt}^*$ and $\ell\in \I_d$,  let $B^\ell_k$ be the  $(2N_{\Dx}+1)^d\times (2N_{\Dx}+1)^d$ real matrix  with entries  given by
\be
\label{matrice_B_K}
(B_{k}^\ell)_{i,j}=\int_{\OO_{\Delta}  } \beta_i( y^\ell_k(x)) \beta_j(x)\dd x \quad \mbox{for} \quad (i,j)\in  \I_{\Dx}\times\I_{\Dx}.
\ee
 Let   $m_{0,\Delta x}$ be the $(2N_{\Dx}+1)^d$ dimensional  real vector with entries   
$$( m_{0,\Delta x} )_i=\int_{\OO_{\Delta}} m^*_0(x)\beta_{i}(x) \dd x \quad \text{for } i\in \I_{\Delta x}.$$ 
Calling $m_{k}=(m_{k,i})_{i\in \I_{\Dx}}$,   scheme \eqref{LG}  can be rewritten in the following matrix form:  find $m_k$ ($k\in \I_{\Dt}$) such that
\be
\label{LG_matrix}
\ba{rrl}\ds 
A m_{k+1}&=& \ds \sum_{\ell\in \I_d} \omega_\ell  B^ \ell_k m_{k} \quad \text{for } k\in \I^*_{\Delta t},\\[12pt]
A m_0&=& m_{0,\Delta x}.
\ea
\ee

\subsection{Properties of the space-time Lagrange-Galerkin scheme}
We show below  some important properties of the scheme \eqref{LG}.
\begin{theorem}
\label{Prop:stability}  
Assume {\bf(H1)},{\bf(H3)}. Then for fixed $\Delta$, there exists a unique solution $(m_{k,i})_{k\in \I_{\Delta t}, i \in \I_{\Delta x}}$  to \eqref{LG_matrix} and,  defining $m_{\Delta}$ as in \eqref{m_Delta},  the following  hold: \medskip\\
{\rm(i)}{\rm[Initial condition]}    $\| m_{0}^*-m_{\Delta }(0,\cdot)\|_{L^2}= O((\Delta x)^{q+1})$ if $m_0^*\in H^{q+1}(\RR^d)$. \medskip\\
{\rm(ii)}{\rm[Mass conservation]}  $\int_{\RR^d} m_{\Delta}(t_{k},x) \dd x=1$ for $k\in \I_{\Dt}$. \medskip\\
{\rm(iii)}{\rm[$L^2$-stability]}   If  $b(t,\cdot)$ is differentiable for all $t\in [0,T]$, then $\max_{k\in \I_{\Dt}} \|m_{\Delta}(t_k,\cdot) \|_{L^2}$ is uniformly bounded with respect to $\Delta$ for $\Delta t$ small enough. 
\end{theorem}
\begin{proof}
	
The well-posedness of \eqref{LG_matrix}  follows from the positive definiteness of $A$ (see e.g. \cite[Proposition 6.3.1]{quarteronivalli94}) and assertion {\rm(i)} is proven in  \cite[Section 3.5]{quarteronivalli94}. In order to prove {\rm(ii)}, fix $k\in \I^*_{\Delta t}$ and sum over $i\in\ZZ^d$ in the first equation of \eqref{LG} to obtain
$$
\sum_{j\in \I_{\Dx }}m_{k+1,j} \sum_{i\in \ZZ^d}\int_{\OO_{\Delta} } \beta_{j}(x)\beta_{i}(x)  \dd x= \sum_{j\in \I_{\Dx }}m_{k,j}\sum_{\ell \in \I_d} \omega_\ell   \sum_{i\in \ZZ^d} \int_{\OO_{\Delta} } \beta_{j}(x) \beta_{i}(y_k^\ell(x))\dd x.
$$
Recalling that, for every $y\in \RR^d$, $\sum_{i\in \ZZ^d} \beta_i(y)=1$, the cardinality $\{ i \in \ZZ^d \, | \, \beta_{i}(y)\neq 0\}$ is bounded uniformly in $y$,  and $\sum_{\ell\in \I_{d}} \omega_\ell=1$, Fubini's theorem yields 
\be\label{integrali}\ba{rcl}
\int_{\OO_{\Delta}}m_{\Delta}(t_{k+1},x) \dd x&=&\sum_{j\in \I_{\Delta x}}m_{k+1,j} \int_{\OO_{\Delta}} \beta_{j}(x) \dd x\\[6pt]
\, &=& \sum_{j\in \I_{\Delta x}}m_{k,j} \int_{\OO_{\Delta}} \beta_{j}(x) \dd x\\[6pt]
\, &=& \sum_{j\in \I_{\Delta x}}m_{0,j} \int_{\OO_{\Delta}} \beta_{j}(x) \dd x\\[6pt]
\, &=& \int_{\OO_{\Delta}}m_{\Delta}(0,x) \dd x.
\ea\ee
Analogously, using the second equation in \eqref{LG} and summing over $i\in \ZZ^d$, we get that
\be\label{integrali_1}
\int_{\OO_{\Delta}}m_{\Delta}(0,x) \dd x= \int_{\OO_{\Delta}} m_0^*(x) \dd x=1.
\ee
Assertion {\rm(ii)} follows from \eqref{integrali}, \eqref{integrali_1},   and \eqref{m_Delta}.
Finally, let us show assertion {\rm(iii)}.  
For $k=0$,  {\rm(iii)} follows from Assumption~{\bf(H1)} and Theorem~\ref{Prop:stability}{\rm(i)}.
For $k\in \I^*_{\Delta t}$,   \eqref {LG} implies that
\be \ba{rcl}\label{eq_stimal2}
\|m_{\Delta}(t_{k+1},\cdot)\|_{L^{2}}^{2} &=& \sum_{\ell\in\I_d}\omega_{\ell}\sum_{i,j\in\I_{\Dx}}m_{k+1,i}m_{k,j}\int_{\OO_{\Delta}}\beta_{i}(x)\beta_{j}(y_k^\ell(x)) \dd x\\[6pt]
&=&\sum_{\ell \in \I_{d}}\omega_l\int_{\OO_{\Delta}}m_{\Delta}(t_{k}, y^\ell_k(x)) m_{\Delta}(t_{k+1},x) \dd x,
 \ea\ee
 and hence, by  the Cauchy-Schwarz inequality, 
\be\label{estimate_m_y}
\|m_{\Delta}(t_{k+1},\cdot)\|_{L^{2}} \leq \max_{\ell \in \I_d} \left(\int_{\OO_{\Delta}} |m_{\Delta}(t_{k}, y^\ell_k(x))|^2 \dd x\right)^{1/2}.
\ee 
In order to estimate the right-hand-side above, fix $x\in \RR^d$,  $\ell\in \I_d$, and  notice that
\be\label{derivative_of_y_k}
Dy^\ell_k(x)= I_{d} + \frac{\Delta t}{2}\bigg( D b(t_{k},x)+Db(t_{k+1},y^\ell_k(x))Dy^\ell_k(x)\bigg),
\ee
where $I_d$ denotes the $d\times d$ identity matrix. Since $Db(\cdot,\cdot)$ is bounded,  there exists   $\ov{\Delta t}>0$ such that for all $k\in \I_{\Dt}^*$   and $\Delta t \in [0,\ov{\Dt}]$, $y^\ell_k$ is one-to-one, and, for all $z\in \RR^d$, the matrix $I_{d} - \frac{\Delta t}{2}Db(t_{k+1},z)$ is invertible. Therefore,  by \eqref{derivative_of_y_k}, 
\be\label{derivative_of_y_k_bis}
Dy^\ell_k(x)= \left(I_{d} - \frac{\Delta t}{2}Db(t_{k+1},y^\ell_k(x))\right)^{-1} \left(I_{d} + \frac{\Dt}{2}Db(t_{k},x)\right),
\ee
from which we deduce that $Dy^\ell_k(x)$ is invertible.  Then, by the change of variable formula, we get that 
\be\label{estimate_l2_norm_k+1_with_Phi}
\int_{\OO_{\Delta}} |m_{\Delta}(t_{k}, y^\ell_k(x))|^2 \dd x= \int_{ y^\ell_k(\OO_{\Delta})} |m_{\Delta}(t_{k}, z)|^2 \big|\mbox{det}  \left(  Dy^\ell_{k}((y^\ell_{k})^{-1}(z)) \right)\big|^{-1} \dd z.
\ee
On the other hand, by \eqref{derivative_of_y_k_bis} and Jacobi's formula, for all $x\in \RR^d$ we have
\be\ba{rcl}
 \left[ {\mbox{det}} \left(  Dy^\ell_{k}(x) \right)\right]^{-1}&=& \ds \frac{\text{det} \left(I_{d}  - \frac{\Delta t}{2}D b(t_{k+1},y ^\ell_k(x))\right)}{ \text{det}   \left( I_{d}+ \frac{\Dt}{2}D b(t_{k},x \right))} \\[20pt]
\; &=& \ds \frac{1- \frac{\Dt}{2} \text{Tr}\left(D b(t_{k+1},y^\ell_k(x)) \right)+O((\Dt)^2)}{1 + \frac{\Dt}{2}\text{Tr}\left(D b(t_{k},x) +O((\Dt)^2\right)}\\[20pt]
\; &=& \ds \frac{1 - \frac{\Dt}{2} \text{div} \left( b(t_{k+1},y^\ell_k(x))\right)+\ O((\Dt)^2) }{1+ \frac{\Dt}{2}\text{div}\left( b(t_{k},x) \right)+ O((\Dt)^2)}.
\ea
\ee
Thus, there exists a constant $C>0$, independent of $x$,  $k$, $\ell$,  and  $\Delta t$, such that
  \be
\big|\left[\mbox{det}  \left( Dy^\ell_{k}(x)\right)\right]^{-1}\big|\leq  1+C\Dt.
\ee
Combining the previous inequality with  \eqref{estimate_l2_norm_k+1_with_Phi}  yields 
\be\label{estimate_aux_1}
\int_{\OO_{\Delta}} |m_{\Delta}(t_{k}, y^\ell_k(x))|^2 \dd x  \leq  (1+C\Dt) \|m_{\Delta}(t_{k},\cdot)\|_{L^2}^{2},
\ee 
and hence, by  \eqref{estimate_m_y}, 
$$
\|m_{\Delta}(t_{k+1},\cdot)\|_{L^{2}} \leq  (1+C\Dt)^{\half} \|m_{\Delta}(t_{k},\cdot)\|_{L^2}^{2}.
$$
Thus,
$$
\|m_{\Delta}(t_{k+1},\cdot)\|_{L^{2}}\leq \left(1+ \frac{CT}{N_{\Delta t}}\right)^{N_{\Delta t}/2}\|m_{\Delta}(0,\cdot)\|_{L^{2}}\leq e^{CT/2}\|m_{\Delta}(0,\cdot)\|_{L^{2}},
$$
from which assertion {\rm(iii)} follows.
\end{proof}

\begin{remark}\label{uniform_integrability_over_compact_sets} Notice that Proposition~\ref{Prop:stability}{\rm(iii)} and the Cauchy-Schwarz inequality imply that, for any compact set $K\subseteq \RR^d$, there exists $C_{K}>0$, independent of $\Delta$ for $\Delta t$ small enough, such that 
$$
\max_{k\in \I_{\Delta t}}    \int_{K} |m_{\Delta }(t_k,x)| \dd x      \leq  C_{K}. 
$$
\end{remark} \vspace{0.3cm}

In the following, we still denote by $m_{\Delta}$  its  extension to $[0,T]\times \RR^d$, defined as  
\be\label{m_extension}
m_{\Delta}(t,x)= \frac{t-t_k}{\Delta t} m_{\Delta}(t_{k+1},x) + \frac{t_{k+1}-t}{\Delta t} m_{\Delta}(t_k,x) \quad \text{if $(t,x) \in [t_k, t_{k+1}]\times \RR^d$ ($k\in \I_{\Dt}^*$)}. 
\ee 
Notice that \eqref{m_extension} and  Theorem~\ref{Prop:stability}{\rm(ii)}-{\rm(iii)} imply that
\be\label{integrale_1_stima_l2}
\int_{\OO_{\Delta}  } m_{\Delta}(t,x) \dd x=1 \quad \text{for all } t\in[0,T] \quad \text{and} \quad \max_{t\in [0,T]} \|m_{\Delta}(t,\cdot) \|_{L^2} \leq C,
\ee
for some $C>0$, independent of $\Delta$  for $\Delta t$ small enough.

For $k\in \NN\cup\{\infty\}$, we denote by $C^{k}_{0}(\RR^d)$ the set of functions of class $C^{k}$ with compact support. 
\vspace{0.05cm}
\begin{proposition}
\label{thm:equicontcons} 
Under {\bf(H1)}-{\bf(H3)}, the following  hold: \smallskip\\
{\rm(i)}{\rm[Equicontinuity]} Let $\phi \in C_0^{q+1}(\RR^d)$. Then there exists $C_\phi>0$ such that for all $\Delta$, with $\Delta t$ small enough and  $(\Delta x)^{q+1}\leq \Delta t$, we have  
\be\label{eq:prop_equicontinuity}
\left|\int_{\RR^d}\phi(x)m_{\Delta}(t,x) \dd x - \int_{\RR^d}\phi(x)m_{\Delta}(s,x)  \dd x\right| \leq  C_{\phi}|t-s|\quad \text{for all $s, \, t\in [0,T]$}.
\ee
{\rm(ii)}{\rm[Consistency]} Assume that  $b(t,\cdot)\in C^{q+1}(\RR^d)$ for all $t\in[0,T]$ and let $\phi \in C^\infty_0(\RR^d)$. Then for any $k \in \I^*_{\Delta t}$ and   $\Delta$, with $\Delta t$ small enough and  $(\Delta x)^{q+1}\leq \Delta t$, we have 
\be\label{prop:consistency_eq}
\ba{l}
\ds \int_{\RR^d}\phi(x) \left(m_{\Delta}(t_{k+1},x) \right.-\left. m_{\Delta}(t_{k},x)\right)\dd x  
= \ds \int_{t_k}^{t_{k+1}} \int_{\RR^d}\left( \frac{\sigma^2}{2} \Delta \phi(x)+\langle b(s,x),\nabla \phi(x)\rangle\right)m_{\Delta}(s,x) \dd x \dd s	 \\[13pt]
\hspace{6.3cm} \ds +O\left(({\Delta x})^{q+1}+({\Delta t})^{2}+\Delta t \omega_{\phi}(\Delta t)\right),
\ea
\ee
where $\omega_{\phi} : [0,\infty)\to \RR$ is a modulus of continuity of $b$ on $[0,T]\times \supp(\phi)$.
\end{proposition}
The proof is given in the Appendix.
Let us denote by  $\D'(\RR^d)$ the space of distributions, which we endow with the weak$^*$ topology. In the following,  for every $ \Delta=(\Dt, \Dx) \in (0,\infty)^2$ and  $t\in [0,T]$, we identify $m_{\Delta}(t, \cdot)$ with the map
$$
C_{0}^{\infty}(\RR^d) \ni \phi \mapsto \int_{\RR^d}\phi(x)m_{\Delta }(t,x) \dd x \in \RR,
$$
which, by Remark~\ref{uniform_integrability_over_compact_sets}, is a regular distribution.  For every $\Delta$, let us denote, with a slight abuse of notation, $m_{\Delta}$ the map  $[0,T]\ni t \mapsto m_{\Delta}(t,\cdot) \in \D'(\RR^d)$.  Notice that Proposition~\ref{thm:equicontcons}{\rm(i)} implies that $m_{\Delta}\in C([0,T];\D'(\RR^d))$. 

\begin{lemma}\label{compatezza_in_D_primo}  Suppose that{ \bf(H1)}{\rm(i)},{\bf(H2)}  hold. Then there exists $\Delta t_0>0$ such that the family 
$\mathcal{M}=\{m_{\Delta} \, | \, \Delta t\leq \Delta t_0, \;  (\Delta x)^{q+1}\leq \Delta t \}$ is relatively compact in $C([0,T];\D'(\RR^d))$.
\end{lemma}

The proof is given in the Appendix.
We conclude the section with the following  convergence result for scheme \eqref{LG}.  
\begin{theorem}
\label{weakconvergence}  
Assume {\bf(H1)} and {\bf(H3)},    $m_{0}^*\in H^{q+1}(\RR^d)$,  $b$ bounded and $ b(t,\cdot)\in C^{q+1}(\RR^d)$ for all $t\in[0,T]$.  Consider a sequence   $\left( \Delta_n \right)_{n\in \NN} = \left((\Delta t_n,\Delta x_n)\right)_{n\in \NN} \subseteq (0,\infty)^2$ such that, as $n\to \infty$, $\Delta_n \to (0,0)$ and  ${(\Delta x}_{n})^{q+1}/{\Delta t}_n \to 0$. Setting  $m^{n}:= m_{\Delta_n}$,  as $n\to \infty$ we have that $(m^n)_{n\in \NN}$  converges to $m^*$ in $C([0,T];\D'(\RR^d))$ and weakly  in  $L^{2}\left([0,T]\times \RR^d \right)$, where $m^*$ is the unique classical solution to \eqref{FP}.
\end{theorem}
\begin{proof}

By Theorem  \ref{Prop:stability}{\rm(iii)}, the sequence $(m^n)_{n\in \NN}$ is bounded in $L^{2}([0,T]\times \RR^d)$. Thus, there exists  $\widehat{m}$ in $ L^{2}([0,T]\times\RR^d)$ such that, as $n\to \infty$ and  up to  some subsequence,  $m^{n}$   converges weakly to $\widehat{m}$ in  $ L^{2}([0,T]\times\RR^d)$.
  
 Let us first show that for any $\phi \in C^{\infty}_0((0,T)\times \RR^d)$, we have
 \be\label{eq:secondweakformula}
\int_{0}^{T}\int_{\RR^d} \left[ \partial_t \phi(t,x)-\frac{\sigma^2}{2} \Delta \phi(t,x) -\langle b(s,x), \nabla \phi (t,x)\rangle \right] \widehat{m}(t,x) \dd x \dd t= 0.
\ee
 Let $\eta \in C^{\infty}_{0}([0,T])$, $\psi \in C^{\infty}_{0}(\RR^d)$ and suppose that $\phi$ has the form $\phi= \eta \psi  \in C_{0}^{\infty}([0,T]\times\RR^d)$. Denote by $K\subset \RR^d$ the support of $\psi$.  By \eqref{m_extension} and Proposition~\ref{thm:equicontcons}{\rm(i)}, we have
 \be\label{primo_svilupo}
\ba{l}
\int_{0}^{T}\int_{\RR^d} \partial_t \phi(t,x) m^{n}(t,x) \dd x \dd t= \ds\sum_{k=0}^{N_{\Delta t_n -1}}\int_{t_k}^{t_{k+1}} \int_{K} \partial_t \phi(t,x) m^{n}(t_k,x) \dd x \dd t  \\[10pt]
\hspace{4.85cm}\ds+\sum_{k=0}^{N_{\Delta t_n -1}}\int_{t_k}^{t_{k+1}} \int_{K} \partial_t \phi(t,x) ( m^{n}(t_{k+1},x) - m^{n}(t_k,x))\frac{t - t_{k}}{\Delta t_n}\dd x \dd t \\[10pt]
\hspace{4.5cm} = \ds\sum_{k=0}^{N_{\Delta t_n -1}}\int_{t_k}^{t_{k+1}} \int_{K} \partial_t \phi(t,x) m^{n}(t_k,x) \dd x \dd t  + O (\Delta t_n).
\ea \ee
On the other hand, by Remark~\ref{uniform_integrability_over_compact_sets} we have
\be\label{second_svilupo}
\ba{rl}
\ds\sum_{k=0}^{N_{\Delta t_n -1}}\int_{t_{k}}^{t_{k+1}} \int_{K} \partial_t \phi(t,x) m^{n}(t_{k},x) \dd x \dd t 
&= \ds\sum_{k=0}^{N_{\Delta t_n -1}} \Delta t_n \int_{K} \partial_t \phi(t_{k},x) m^{n}(t_{k},x) \dd x   + O\left( \Delta t_n \right)\\[10pt]
&=\ds\sum_{k=0}^{N_{\Delta t_n}-1}\Delta t_{n}\dot{\eta}(t_k)\int_{K}\psi(x) m^{n}(t_k,x) \dd x + O(\Delta t_n)\\[10pt]
&=\ds\sum_{k=0}^{N_{\Delta t_n}-1}(\eta(t_{k+1}) -\eta(t_k))\int_{K}\psi(x) m^{n}(t_k,x) \dd x + O(\Delta t_n)\\[10pt]
&= \ds\sum_{k=0}^{N_{\Delta t_n}-2}\eta(t_{k+1})\left(\int_{K} \psi(x)[m^{n}(t_k,x)- m^n(t_{k+1},x)] \dd x  \right) \\[18pt]
& \; \; \; \; + O(\Delta t_n).
\ea
\ee
By \eqref{primo_svilupo}, \eqref{second_svilupo} and using that $\psi$ vanishes outside $K$, we get
\be\label{terzosviluppo}\ba{rcl}
\ds \int_{0}^{T}\int_{\RR^d} \partial_t \phi(t,x) m^{n}(t,x) \dd x \dd t&=& \ds \ds\sum_{k=0}^{N_{\Delta t_n}-2}\eta(t_{k+1})\left(\int_{\RR^d} \psi(x)[m^{n}(t_k,x)- m^n(t_{k+1},x)] \dd x  \right)\\[15pt]
\; & \; & + O(\Delta t_n).
\ea
\ee
Using \eqref{terzosviluppo}, Proposition~\ref{thm:equicontcons}{\rm(ii)}, and  Remark~\ref{uniform_integrability_over_compact_sets}, we have
$$
\ba{l} 
\ds \int_{0}^{T}\int_{\RR^d} \partial_t \phi(t,x) m^{n}(t,x) \dd x \dd t=\ds \sum_{k=0}^{N_{\Delta t_n}-1} \eta(t_{k+1})\int_{t_k}^{t_{k+1}} \int_{\RR^d}\left( \frac{\sigma^2}{2} \Delta \psi (x) +\langle b(s,x), \nabla \psi(x)\rangle \right)m^{n}(s,x) \dd x \dd s \\[18pt]
\hspace{5cm} +O(({\Delta x}_{n})^{q+1}/\Delta t_n+{\Delta t}_{n} +\omega(\Delta t_n))\\[10pt]
\hspace{4.6cm}= \ds \int_{0}^{T}\int_{\RR^d}\left(  \frac{\sigma^2}{2} \Delta \phi(t,x) +\langle b(s,x) , \nabla \phi (t,x)\rangle\right)m^{n}(t,x) \dd x \dd t \\[18pt]
\hspace{5cm}+O(({\Delta x}_{n})^{q+1}/\Delta t_n+{\Delta t}_{n}+\omega(\Delta t_n)),
\ea$$
where $\omega: [0,\infty)\to \RR$ is a modulus of continuity of $b$ on $[0,T]\times K$. Thus, 
$$\int_{0}^{T}\int_{\RR^{d}} \left[ \partial_t \phi(t,x) -\frac{\sigma^2}{2} \Delta \phi(t,x)-\langle b(s,x), \nabla \phi (t,x)\rangle\right] m^{n}(t,x) \dd x \dd t= O(({\Delta x}_{n})^{q+1}/\Delta t_n+{\Delta t}_{n}+\omega(\Delta t_n))
$$
and hence, passing to the weak limit in $L^2([0,T]\times \RR^d)$, we get  
\be
\int_{0}^{T}\int_{\RR^d} \left[ \partial_t \phi(t,x) -\frac{\sigma^2}{2} \Delta \phi(t,x)-\langle b(s,x), \nabla \phi (t,x)\rangle\right] \widehat{m}(t,x) \dd x \dd t= 0.
\ee
Since the vector space spanned by  $\{ \eta \psi \; |  \; \eta \in C^{\infty}_{0}((0,T)), \; \psi \in C^{\infty}_{0}(\RR^d) \}$ is dense in $C^{1,2}_0((0,T)\times \RR^d)$ (as in \cite[Corollary 1.6.2 of the Weierstrass Approximation Theorem]{narasimhan1973analysis}), 
we get that \eqref{eq:secondweakformula} holds for any $\phi \in C_{0}^{1,2}((0,T)\times \RR^d)$. 

Finally, let us show that  for any $\phi \in C_{0}(  \RR^d)$
 \be \label{eq:datoiniziale}
 \int_{\RR^d}\phi(x) (\widehat{m}(t,x)-m_0^*(x))\dd x\to0\quad {\mbox as\;}t\to 0^+.\ee
 
By Lemma \ref{compatezza_in_D_primo},  we have that $\widehat{m} \in C([0,T];\D'(\RR^d))$. Moreover, by \cite[Lemma 2.1]{figalli08}, for any $t\in [0,T]$ and  for every $\phi\in C_{0}(\RR^d)$, it holds that
\be\label{continuous_representative}
\lim_{ s\to t, \, s\in [0,T]} \int_{\RR^{d}} \phi(x) \widehat{m}(s,x) \dd x =  \int_{\RR^{d}} \phi(x) \widehat{m}(t,x) \dd x.
\ee
 
Since  Theorem~\ref{Prop:stability}{\rm(i)} implies that $\widehat{m}(0,\cdot)= m_0^*(\cdot)$,
\eqref{eq:datoiniziale} follows from  \eqref{continuous_representative} with $t=0$. Thus, the result follows from \eqref{eq:secondweakformula}, \eqref{eq:datoiniziale} and the uniqueness result in Theorem~\ref{teorema_well_posedness}{\rm(iv)}.
\end{proof}

\begin{remark} The convergence of the sequence $(m^n)_{n\in \NN}$ to $m^*$ in the previous theorem is rather weak. On the other hand,  to the best of our knowledge this is the first convergence result of a high-order LG scheme for equation \eqref{FP}. Notice that our proof does not depend on the smoothness of $m^*$ recalled in Theorem~\ref{teorema_well_posedness}{\rm(i)},  and it can be easily adapted to deal with equations whose second-order term are not uniformly elliptic {\rm(}see e.g. \cite{ferrettisecondoordine,MR3828859}{\rm)}.
\end{remark}

\section{The scheme  for MFG} \label{sec_schemeMFG} 
To derive a high-order scheme that approximates a solution $(v^*,m^*)$ of the MFG system, we are left to derive a high-order method for the HJB equation, which coupled with \eqref{LG} will provide the desired discretization of system \eqref{MFG}.\\
Given $\mu \in C([0,T];\mathcal{P}_1(\RR^d))$,  we consider the HJB equation:
\be
\label{hjbmu}
\tag{{\bf HJB}}
\ba{rcl} -\partial_{t} v   - \frac{\sigma^2}{2} \Delta v +\half |\nabla v |^{2}   & =& F(x, \mu(t))  \hspace{0.4cm}  \hbox{in }   (0,T)\times \RR^d,  \\[6pt]	
										v( T,\cdot) & =& G(\cdot,\mu(T)) \hspace{0.3cm}  \hbox{in }  \RR^{d}.\ea
\ee
Standard  results for quasilinear parabolic equations (see e.g.  \cite[Chapter IV and V]{Lady67})  yield that \eqref{hjbmu} admits a unique classical solution $v[\mu]$. Moreover,  using that $v[\mu]$ is the value function associated with a stochastic optimal control problem (see e.g. \cite[Chapters IV and V]{MR2179357}), it is easy to check that {\bf(H1)-\bf(H2)}  imply the existence of $R>0$ such that 
$$|\nabla v[\mu](t,x)|\leq  R \quad \text{for all } t\in [0,T], \, x\in \RR^d, \, \mu \in C([0,T]; \P_1(\RR^d). $$ 

We now describe  a variation of the scheme in \cite{BCCF21} to deal with the nonlinearity of the Hamiltonian in \eqref{hjbmu} with respect to $\nabla v$ (see also \cite{milstein:2000,PicarelliReis2020} for related constructions).   For a  given $ \mu \in C([0,T];\mathcal{P}_1(\RR^d))$,  let us define  $\{v_{k,i} \, | \, k\in \I_{\Delta t}, \, i\in \I_{\Delta x}\} \subset \RR $ as the solution to 
\begin{equation}
\label{scheme-control}
\ba{rcl}
  v_{k,i}&=&S[\mu](v_{\cdot,k+1},k,i) \quad \text{for all } k\in \I^*_{\Delta t}, \, i\in \I_{\Delta x},\\[6pt]
   v_{N_{\Delta t},i}&=& G(x_{i}, \mu(t_{N_{\Delta t}})) \quad   \text{for all }   i\in \I_{\Delta x}, \\
\ea
\end{equation}
where, for a given  $f=\{f_i\}_{i\in \I_{\Dx}}\subset \RR$, $k\in \I_{\Delta t}^*$, and $i\in \I_{\Delta x}$,  
\begin{equation}
\label{SLscheme}
\begin{split}
S[\mu](f,k,i)=\inf_{\alpha \in A} \left[\sum_{\ell\in \I_{d}} \ \omega_{\ell} \left(I[f](x_i -\Delta t \alpha+\sqrt{\Delta t}\sigma e^\ell) +  \frac{\Delta t}{2} F(x_i -\Delta t \alpha+\sqrt{\Delta t}\sigma e^\ell, \mu(t_{k+1}))\right) \right. \\
\left. + \frac{\Delta t}{2} |\alpha|^{2} \right] + \frac{\Delta t}{2} F(x_{i}, \mu(t_{k})),
\end{split}
\end{equation}
with $A= \{\alpha \in \RR^d \, | \, |\alpha| \leq R\}$ and $I[f]$ being defined by \eqref{definizione_I}.  The following consistency result for $S[\mu]$ follows from \eqref{SLscheme} and {\bf (H2)}. 
\begin{proposition}
\label{propiedadesbasicasesquema}   
Let $(\Delta t_n, \Delta x_n)_{n\in \NN}\subset (0,+\infty)^2$, $(k_n)_{n\in \NN}\subseteq \NN$, $(i_n)_{n\in \NN}\subset \ZZ^d$,   $(\mu_n)_{n\in \NN}\subset C([0,T];\P_1(\RR^d))$, and $\mu \in C([0,T];\P_1(\RR^d))$. Assume that {\bf(H2)} holds and, as $n\to \infty$,  $ (\Delta t_n, \Delta x_n) \to(0,0)$,  ${(\Delta x}_{n})^{q+1}/{\Delta t}_n \to 0$, $k_n \in \I_{\Delta t_n}$, $i_n\in \I_{\Delta x_n}$,  $t_{k_n}\to t$, $x_{i_n} \to x$,  and $\mu_n\to \mu$. Then,  for every 
$\phi \in  C_{b}^{1,3}\left([0,T] \times \RR^{d}\right)$   satisfying $\|\nabla \phi \|_{L^\infty([0,T]\times \RR^d)}\leq R$, we have
$$\label{Consistenzadebole}
\underset{n\to \infty} \lim \frac{1}{{\Delta t}_n} \left[\phi(t_{k_n},x_{i_n})-S[\mu_n](\phi_{k_{n}+1},k_n,i_n)\right]  =-\partial_t \phi(t,x)-\frac{\sigma^2}{2}\Delta \phi(t,x) +\frac{1}{2}|\nabla\phi(t,x)|^2-F(x,\mu(t)),
$$
where $\phi_{k}=\{\phi(t_k,x_i)\}_{i\in\I_{\Dx}}$.
\end{proposition}

The proof is given in the Appendix.\\

For $\mu \in C([0,T];\mathcal{P}_1(\RR^d))$, let us define
\be\label{v_extension} v_{\Delta }[\mu](t,x):=I [v_{\left[ t/\Dt \right]} ](x) \hspace{0.5cm} \mbox{for all } \hspace{0.2cm} (t,x) \in [0,T] \times \OO_{\Delta},
\ee
where $v_{k,i}$ is given by \eqref{scheme-control}.
For $\ell \in\I_d$ and $k \in \I^*_{\Delta t}$, let $y_{k}^\ell[\mu](x)$ be the unique solution to
\be\label{discretecaract_eps} 
y=x- \frac{\Delta t}{2} \left[Dv_{\Delta}[\mu](t_k,x)+Dv_{\Delta}[\mu](t_{k+1},x), y)\right]+\sqrt{\Delta t}\sigma e^\ell,
\ee
where $Dv_{\Delta}[\mu](t_k,x)$ represents a numerical gradient with respect to $x$ of $v_{\Delta}[\mu](t_k,x)$, computed by a fourth-order finite difference approximation.\\

We propose the following scheme for \eqref{MFG}: find $\{(v_{k,i},m_{k,i})\in \RR^2 \, | \, k \in \I_{\Delta t}, \; i\in \I_{\Delta x}\}$ such that, for all $k\in \I_{\Delta t}^*$ and $i\in \I_{\Delta x}$,
\be
\label{weak_scheme_MFG} 
 \ba{rrl}
\ds v_{k,i}&=& \ds S_{\Delta}[m_{\Delta}](v_{k+1},k,i),\\[6pt]
\ds v_{N_{\Delta t},i}&=& \ds G(x_{i},m),\\[6pt]
\ds \sum_{j\in \I_{\Delta x}} m_{k+1,j}\int_{\OO_{\Delta} }\beta_{i}(x)\beta_j(x) \dd x&=& \ds \sum_{j\in \I_{\Delta x}} m_{k,j}\sum_{\ell\in \I_{d}}\omega_\ell \int_{\OO_{\Delta} } \beta_i ( y_{k}^{\ell}[m_{\Delta}](x)) \beta_j(x)\dd x, \\[15pt]
\ds \sum_{j \in \I_{\Delta x}} m_{0,j} \int_{\OO_{\Delta}} \beta_{i}(x) \beta_{j}(x) \dd x &=&\ds \int_{\OO_{\Delta}} m_0^*(x)\beta_{i}(x) \dd x.
\ea\ee
In the examples considered in the following section, system~\eqref{weak_scheme_MFG} will be solved, heuristically and as in  \cite{MR3148086,MR4253925,Jakobsen_et_al_2021}, by using fixed-point iterations.

\section{Numerical results}
\label{sec:numerics}
In this section, we show the performance of the proposed scheme on two different problems: a MFG with non-local couplings and an explicit solution, and a MFG with local couplings and no explicit solutions. For each test, we measure the accuracy of the scheme by computing the following relative errors in the discrete uniform and $L^2$ norms
$$
E_\infty = \frac{\max_{i \in \I_{\Delta x}} | h_{\Delta}(T,x_i) - h(T,x_i)|}{\max_{i \in \I_{\Delta x}} | h(T,x_i) |}, \quad E_2 = \left(\frac{\mbox{Int}_{\OO_{\Delta}} (| h_{\Delta}(T,x) - h(T,x)|^2 )}{\mbox{Int}_{\OO_{\Delta}}(| h(T,x) |^2 )}\right)^{1/2},
$$
where $h=m,\, v$, $h_\Delta = m_\Delta, \, v_\Delta$, and $\mbox{Int}_{\OO_{\Delta}}$ denotes the approximation of the Riemann integral on $\OO_{\Delta}$ by using the Simpson's Rule. We denote by $p_\infty$ and $p_2$ the rates of convergence for $E_{\infty}$ and $E_2$, respectively.

Notice that, for the exactly integrated scheme \eqref{LG}, the local truncation error is given by the contributions of \eqref{stimaCN} and \eqref{approximation_error_V_Deltax_q}, which yield a global truncation error of order $(\Dx)^{q+1}/\Dt + (\Dt)^2$. As in \cite{ferrettisecondoordine}, we get that the order of consistency is maximized by taking $\Delta t = O((\Delta x)^{(q+1)/3})$. With respect to the space discretization step, the previous choice suggests an order of convergence given by $2(q+1)/3$. In all the simulations we take $q=3$, which yields an heuristic optimal rate equal to $8/3$, and Simpson's Rule to approximate the integrals in \eqref{LG}. The convergence rate of the resulting scheme is illustrated numerically in the examples below. Indeed, the tables in the tests show rates of convergence $p_\infty$ and $p_2$ greater than $2$ in most of the cases.
The positivity preservation of the discrete density is true only when linear basis functions are used. This property is not in general verified by our method, and it holds only asymptotically.
In the first numerical test, we calculate the maximum value of the negative part of the approximate density in the space-time mesh and call this value positivity error. We will show that with the refinement of the mesh, the positivity error decreases until it cancels.
\subsection{ An implementable version of the scheme \eqref{LG}}
\label{fullyd_1d}
In order to obtain an implementable version of \eqref{LG_matrix}, an approximation of the integrals therein has to be introduced. For simplicity, we consider the one-dimensional case, we use Simpson's Rule on each element
$[x_j,x_j+2\Dx]$ ($j=2m$, $m\in \mathbb{Z}$) and cubic symmetric Lagrange interpolation basis functions $\beta_j$ ($p=1$ in \eqref{eq:referencebasis}). Recalling that $\beta_{j}$ has support in $[x_{j-2},x_{j+2}]$, letting $\delta_{i,j}=1$ if $i=j$ and $\delta_{i,j}=0$ otherwise, the entries of the mass matrix $A$ (see \eqref{def_mass_matrix}) are approximated by
\be
\label{approx_int_1}
\int_{\OO_{\Delta}} \beta_{i}(x) \beta_{j}(x) \dd x=
\int_{x_{j-2}}^{x_{j}} \beta_{i}(x) \beta_{j}(x) \dd x+\int_{x_j}^{x_{j+2}} \beta_{i}(x) \beta_{j}(x) \dd x\simeq
\frac{2 \Dx}{3} \delta_{i,j},
\ee
while the entries of $B_{k}^{\ell}$ (see \eqref{matrice_B_K})  are approximated by
\be
\label{approx_int_2}
(B_{k}^\ell)_{i,j}=\int_{x_{j-2}}^{x_{j+2}} \beta_{ i}(y^\ell_k(x))\beta_j(x)\dd x\simeq \frac{2 \Dx}{3}  \beta_{ i}(y^\ell_k(x_j)).
\ee
We observe that, as usual in LG methods, the integrands in \eqref{approx_int_1} and \eqref{approx_int_2} have not the necessary regularity in order to guarantee the standard accuracy order of the quadrature rule.
This can lead to fluctuations in the order of convergence, as can be observed in some instances of the numerical tests below. However, in those tests we will see that the aforementioned quadrature rule provides an overall order of convergence  close to $8/3$.

Using \eqref{approx_int_1} and \eqref{approx_int_2}, the scheme \eqref{LG_matrix}  is approximated by  
\be
\label{LGdiscreto}
\ba{rcl}
m_{k+1}&=&\sum_{\ell\in \I_{d}} \omega_\ell \widetilde B_k^ \ell m_{k} \quad \text{for }k \in \I^*_{\Delta t},\\[6pt]
m_0&=& \widetilde{m}_0,
\ea
\ee
where $\widetilde B_k^\ell$  is a $(2N_{\Dx}+1)\times (2N_{\Dx}+1)$ matrix with entries given by 
$$
(\widetilde{B}_{k}^\ell)_{i,j}=\beta_{ i}(y^\ell_k(x_j))
$$
and $\widetilde{m}_0$ is vector of length $2N_{\Dx}+1$ given by 
$$
\widetilde{m}_{0,i}= m_{0}^*(x_i) \quad \text{for } i \in \I_{\Delta x}.
$$
\begin{remark} Applied to a linearization of equation  \eqref{hjbmu}, scheme \eqref{LGdiscreto} is the dual of the semi-Lagrangian scheme \cite{ferrettisecondoordine} when a Crank-Nicolson method is used to discretize the characteristic curves, together with a cubic symmetric Lagrange interpolation to reconstruct the values in the space variable. Moreover,  scheme \eqref{LGdiscreto} is also a natural higher-order extension of the scheme proposed in \cite{CS15,MR3828859} to approximate second-order MFGs. 
\end{remark}

\subsection{Fixed-point iterations}\label{sect:alg}
In view of~\eqref{LGdiscreto}, it is natular to propose the following implementable version of the scheme for \eqref{MFG}:
 find $(\mathcal V,\mathcal M):=\{(\overline v_{k,i},\overline m_{k,i})\in \RR^2 \, | \,  k \in \I_{\Delta t}, \; i\in  \I_{\Delta x}\}$ such that, for all $k\in \I_{\Delta t}^*$ and $i\in \I_{\Delta x}$,
\begin{align}
 v_{k,i}&= \ds S_{\Delta}[m_{\Delta}](v_{k+1},k,i), \label{scheme_HJ}\\
 v_{N_{\Delta t},i}&= \ds G(x_{i},m), \nonumber \\
 m_{k+1,i}&=\sum_{j \in  \I_{\Delta x}}m_{k,j} \sum_{\ell\in \I_{d}} \omega_\ell \beta_{ i}(y^\ell_k(x_j)),\label{scheme_FP} \\
 m_{0,i}&= m_{0}^*(x_i)\nonumber.
\end{align}
This system is heuristically solved by the fixed-point iterations described in {\bf{Algorithm 1}}, which has as input data a damping (or relaxation) parameter $\omega \in [0,1]$, an initial guess ${\mathcal M}_\circ$ for the density $\mathcal M$, and a tolerance parameter $\tau>0$. 
{\color{black}
The iterations are stopped as soon as the  $L^1$-norm, approximated by the Simpson's Rule,  of the difference between two consecutive approximations of $\mathcal M$ is less than $\tau$.

\begin{algorithm}
    \SetKwFunction{isOddNumber}{isOddNumber}
    \SetKwInOut{KwIn}{Input}
    \SetKwInOut{KwOut}{Output}
    \SetKwRepeat{Do}{do}{while}
    \KwIn{Initial guess ${{\mathcal M}_\circ}$, damping parameter $\omega$, 
     tolerance $\tau>0$.}
    \KwOut{Approximation $(\mathcal V,\mathcal M)$ of the solution to  \eqref{scheme_HJ}-\eqref{scheme_FP}.}

        Initialize $\widetilde {\mathcal M}^{(0)}= {\mathcal M}^{(0)}={\mathcal M}_\circ$, $p=0$,
        
   \Do{ $E^{(p+1)}>\tau$}
     {
    compute ${\mathcal V}^{(p+1)}$ solution to \eqref{scheme_HJ} with $m_{\Delta}$ replaced by ${\widetilde {\mathcal M}^{(p)}}$,\\
    compute $D v_{\Delta}^{(p+1)}$ numerical gradient of ${\mathcal V}^{(p+1)}$,\\
   compute   ${\mathcal M}^{(p+1)}$ solution to \eqref{scheme_FP}   with $y^\ell_k$ solution to \eqref{discretecaract_eps}
   obtained using   $D v_{\Delta}^{(p+1)}$,\\
   compute $E^{(p+1)}= \mbox{Int}_{[0,T]\times \OO_{\Delta}}|{\mathcal M}^{(p+1)}-{\mathcal M}^{(p)}|$,\\
  let $\widetilde {\mathcal M}^{(p+1)}=\omega {\widetilde {\mathcal M}^{({p})}}+(1-\omega){ {\mathcal M}^{(p+1)}}$, \\
   set $p=p+1$,\\
   }

     \Return  $ ({\mathcal V}^{(p+1)}, {\mathcal M}^{(p+1)}). $
 
    \caption{Fixed-point iterations}\label{alg}
\end{algorithm}
}

\subsection{Non-local MFG with analytical solution}
\label{non_local_test}
Consider the non-local MFG system 
\be\label{test2}
\ba{rcl}
-\partial _t v - \frac{\sigma^2}{2}\Delta v + \half \left| \nabla v  \right|^2 &=& \ds \half \left( x - \int_{\RR^d}y m(t,y) \mathrm{d}y\right)^2 
\quad  \text{in } [0,T) \times \RR^d, \\[10pt]
\partial_t m  - \frac{\sigma^2}{2} \Delta m  - \divergence\left(\nabla v m\right) &=&0 \quad  \text{in }(0,T] \times \RR^d, \\[5pt]
v(T,\cdot) = 0, \; &\;& \;  m(0,\cdot) =m_0^*    \quad  \text{in } \RR^d,
\ea
\ee
where $m_0^*$ is the density of a Gaussian random vector with mean $\mu_0\in \RR^d$ and covariance matrix $\Sigma_0 \in \RR^{d\times d}$. For simplicity, we will assume that  $\Sigma_0$ is a diagonal matrix.

In what follows, we  compute explicitly the unique solution $(v^*,m^*)$ to \eqref{test2} (see e.g. \cite{MR3489817}).  Since $v^*$ is the value function associated with a linear-quadratic optimal control problem, standard results (see e.g. \cite[Chapter 6]{MR1696772}) show that $v^*$ has the form 
\be
\label{v_star_esempio}
v^*(t,x)= \half \langle \Pi(t)x, x \rangle +\langle s(t),x\rangle + c(t)\quad \text{for $(t,x)\in [0,T]\times \RR^d$},
\ee
where, setting $\ov Y(t) = \int_{\RR^d} y m^*(t,y) \mathrm{d}y$ for all $t\in [0,T]$,  $\Pi$, $s$, and $c$ satisfy
\be
\label{eq:exactsol}
\begin{array}{ll}
- \dot \Pi(t)  = - \Pi^{2}(t) + I_{d} &\quad \text{for $t\in(0,T)$},\\[4pt]
-\dot s(t)  \hspace{0.1cm}= - \Pi (t)s(t)- \ov Y(t)&\quad \text{for $t\in(0,T)$}, \\[4pt]
-\dot c(t)  \hspace{0.1cm}= \frac{\sigma^2}{2}\text{Tr}(\Pi(t)) -\half |s(t)|^2+ \half\left| \ov Y(t)\right|^2 & \quad \text{for $t\in(0,T)$}, \\[4pt]
\Pi(T) \hspace{0.1cm}= 0, \quad
s(T) = 0, \quad
c(T) = 0.
\end{array}
\ee
Notice that   $\Pi$  satisfies a Riccati equation whose analytical solution is given by
\be
\label{sol_pi}
\Pi (t) = \left(\frac{e^{2T-t} - e^t}{e^{2T-t} + e^t}\right)  I_{d}\quad \text{for $t\in [0,T]$.}
\ee
Since  $\nabla v^*(t,x)=  \Pi(t) x + s(t)$, the SDE underlying the FP equation in \eqref{test2} (see \eqref{SDE_underlying_FP}) is given by
$$
\ba{rcl}
\mathrm{d}Y(t) &=& \left[-\Pi(t) Y(t) - s(t)\right] \dd  t + \sigma \mathrm{d}W(t) \quad \text{for $t\in (0,T)$}, \\[6pt]
Y(0)&=& Y_0,
\ea
$$
where $Y_0$ is a  Gaussian random variable,  independent of the $d$-dimensional Brownian motion $W$, with mean $\mu_0$ and covariance matrix $\Sigma_0$.  Since
\be
\label{sde_Y}
Y(t) = Y_0 + \int_0^t \left[-\Pi(r) Y(r) - s(r)\right] \dd r +\sigma W(t) \quad \text{for $t\in (0,T)$},
\ee
and the coordinates $Y_{0,i}$ ($i=1,\hdots,d$) of $Y_0$ are independent Gaussian random variables with means $\mu_{0,i}$ and variance $(\Sigma_0)_{i,i}$, 
for every $t\in [0,T]$, $Y(t)$ is a vector of independent Gaussian random variables $Y_{i}(t)$ ($i=1,\hdots,d$)  with mean $\ov{Y}_{i}(t)$ and variance  $(\Sigma(t))_{i,i}=\EE\left(Y_{i}^2(t)\right) -\ov{Y}^2_{i}(t)$ to be determined.   In other words, 
\be
\label{m_star_prodotto}
m^*(t,x)= \Pi_{i=1}^{d}m^*_{i}(t,x_i)  \quad \text{for } t\in [0,T], \; x\in \RR^d,
\ee
where, for every $t\in [0,T]$ and  $i=1, \hdots, d$, $m_{i}^{*}(t,\cdot)$ is a univariate Gaussian density with parameters $\ov{Y}_{i}(t)$ and variance  $(\Sigma(t))_{i,i}$  In order to compute these parameters, notice that \eqref{sde_Y} implies that 
$$
\ov Y (t)=  \mu_0  + \int_0^t \left(-\Pi(r)\ov Y(r) - s(r)\right)\mathrm{d}r \quad \text{for $t\in [0,T]$},
$$
i.e. $\ov Y$ solves 
\be
\label{eq:barx_wr}
\ba{rcl}
\dot {\ov Y} (t) &=& -\Pi(t) \ov Y(t) - s(t)  \quad  \text{for } t\in(0,T), \\[6pt]
\ov Y (0) &=& \mu_0.
\ea
\ee
Thus, by \eqref{eq:exactsol} and \eqref{eq:barx_wr}, the couple $(\ov Y, s)$ solves the boundary value problem
$$
\ba{ll}
\dot {\ov Y}(t)\hspace{0.025cm}= -\Pi (t)\ov Y(t)- s(t)\quad &\text{for } t\in(0,T), \\[6pt]
\dot s (t) \hspace{0.15cm}= \Pi (t)s(t)+ \ov Y(t) \quad &\text{for } t\in(0,T), \\[6pt]
\ov Y (0) = \mu_0 ,\quad &
s(T) = 0,
\ea
$$
whose unique solution is given by (see e.g. \cite{heath2005scientific})
\be
\label{Y_e_s}
\ov Y(t) = \mu_0,\quad
s(t) = - \Pi(t) \mu_0 \quad \text{for } t\in [0,T],
\ee
where we recall that $\Pi$ is given by \eqref{sol_pi}.

On the other hand, by \eqref{sde_Y} and It\^o's lemma, for every $i=1, \hdots,d$, we have
$$
 Y_{i}^2(t) = Y^2_{0,i} - \int_0^t 2 Y_{i}(r)\left[\Pi_{i,i}(r)Y_i(r) + s_i(r)\right]\mathrm{d}r + 2\sigma\int_0^t Y_i(r) \mathrm{d}W_i(r) + \sigma^2t \quad \text{for } t\in [0,T].
$$
Thus, denoting by $m_{0,i}^*$ the $i$ the marginal of $m_{0}^*$ ($i=1,\hdots,d$), \eqref{Y_e_s} yields
$$
 \EE \left( Y_i^2(t)\right) = \int_{\RR} x^2 m_{0,i}^*\left(x\right)\dd x - 2\int_0^t \left[\Pi_{i,i}(r) \EE \left( Y_{i}^2(r)\right)+ \mu_{0,i}s_i(r)\right]\mathrm{d}r + \sigma^2t \quad \text{for } t\in [0,T].
$$
In particular, $[0,T]\ni t \mapsto \EE \left( Y_i^2(t)\right) \in \RR $ is the unique solution to
$$
\ba{rcl}
\dot M(t) &=& -2\Pi_{i,i}(t)M(t) - 2 \mu_{0,i} s_i(t) + \sigma^2 \quad t\in(0,T),\\[6pt]
 M(0) &=& \int_{\RR} x^2 m_{0,i}^*\left(x\right)\mathrm{d} x,
\ea
$$
which, for all $t\in [0,T]$, is given by
 \be
\EE \left( Y_{i}^2(t)\right) = \ds \left(e^{2T - t} + e^t\right)^2\left[ \ds \frac{2 \int_{\RR} x^2 m_{0,i}^*\left(x\right)\mathrm{d}x - 2 \mu_{0,i}^2 + \sigma^2 \left(e^{2T} + 1\right)}{2(e^{2T} + 1)^2} - \frac{\sigma^2}{2(e^{2T} + e^{2t})} \right] + \mu_{0,i}^2.
\ee
Thus, for all $i=1,\hdots,d$ and $t\in [0,T]$,
\be
\label{sigma_t}
(\Sigma(t))_{i,i}= \left(e^{2T - t} + e^t\right)^2\left[ \frac{2 \int_{\RR} x^2 m_{0,i}^*\left(x\right)\dd x - 2 \mu_{0,i}^2 + \sigma^2 \left(e^{2T} + 1\right)}{2(e^{2T} + 1)^2} - \frac{\sigma^2}{2(e^{2T} + e^{2t})} \right].
\ee
Altogether, for all $t\in [0,T]$, $m^*(t,\cdot)$ is given by \eqref{m_star_prodotto}, where the parameters of the univariate Gaussian densities $m_{i}^*(t,\cdot)$ are given by \eqref{Y_e_s} and \eqref{sigma_t}, and the value function $v^*$ is given by \eqref{v_star_esempio}, with $\Pi$ and $s$ given by \eqref{sol_pi} and \eqref{Y_e_s}, respectively, and $c$, obtained by integrating the third equation of \eqref{eq:exactsol}, is given by
$$
c(t) = \half \left\langle \Pi(t)\mu_0,\mu_0\right\rangle - \frac{\sigma^2d}{2} \ln \left(\frac{2e^T}{e^{2T-t} + e^t}\right) \quad \text{for } t\in [0,T].
$$
In this test, the assumption on the boundedness of $b$ is not verified in all $\mathbb{R}^d$, however it is true in every bounded domain $\OO_{\Delta}$.
Let us now solve system \eqref{test2} on a bounded domain in dimension $d=1$. We choose $[0,T] \times \OO_\Delta = [0, 0.25] \times (-2,2)$, with Dirichlet boundary conditions on $\partial\OO_\Delta$, the latter being equal to the exact solution of \eqref{test2} for the HJB equation and homogeneous for the FP equation. The numerical approximation of the boundary conditions for the HJB equation is based on the technique proposed in \cite{BCCF21}, while for the FP equation we proceed as in the previous test.
We consider two cases, one with $\sigma^2/2 = 0.05$ and the other one with $\sigma^2/2 = 0.005$. As input parameters for the fixed-point iterations, we set the initial guess for the density equal to the initial datum $m_0^*$ at each time step, as damping parameter $\omega=0$, and as tolerance
 $\tau=10^{-9}$.
Tables \ref{table:testsolesatta1nuovaV} and \ref{table:testsolesatta1nuova} show the errors and the convergence rates for the approximation of the HJB and FP equations, in the case where $\sigma^2/2 = 0.05$ and $\Dt = (\Dx)^{4/3}/4$.
In order to show the advantages of the high-order scheme in this paper, we also solve system \eqref{test2} with the low order numerical scheme proposed in \cite{CS15}. We display in Table \ref{table:testsolesatta1nuova} (columns 6-9) the errors and convergence rates for the approximation of the FP equation. The comparison between the errors and the orders of the two schemes clearly shows the gain in accuracy
achieved by the new scheme.

Tables \ref{table:testsolesatta2nuovav} and \ref{table:testsolesatta2nuova} show the errors and the convergence rates for the approximation of the HJB and FP equations, in the case where $\sigma^2/2 = 0.005$ and $\Dt = (\Dx)^{4/3}/4$.
The convergence rates tend to be close to the theoretical optimal rate $8/3$. Tables \ref{table:testsolesattadxfratto4v} and \ref{table:testsolesattadxfratto4m} show the errors and convergence rates for $v_{\Delta}(0,\cdot)$ and $m_{\Delta}(T,\cdot)$, which are calculated by taking  $\Delta t=\Delta x/4$, and Tables \ref{table:testsolesattadxquadrov} and \ref{table:testsolesattadxquadrom} consider the case $\Delta t=(\Delta x)^2$.
These tables show that the scheme is stable as the time steps change, however the convergence rate deteriorates slightly, especially for the approximation of the time-dependent density.
 \\
 In Figure \ref{fig:testesattafig1} we show the solution to \eqref{test2} on $[0,T] \times \OO_\Delta = [0,0.25] \times (-2,2)$ with $\sigma^2/2=0.005$, computed with $\Delta x = 1.25 \cdot 10^{-2}$ and $\Delta t = (\Delta x)^{4/3}/4$. Figure \ref{fig:testesatta2zoom} displays a zoom of the initial density $m_0^*$, the exact solution $m^*(T,\cdot)$ and its approximation $m_{\Delta}(T,\cdot)$, computed with $\Delta x = 6.25 \cdot 10^{-3}$ and $\Delta t = (\Delta x)^{4/3}/4$.
\begin{table}[h!]
    \centering
\begin{tabular}{||c|c|c|c|c|} 
\hline
\multirow{2}{*}{$\Delta x$ }&  \multicolumn{4}{c|}{Errors for the approximation of $v^*(0,\cdot)$}\\
\cline{2-5}
& $E_{\infty}$ &  $E_2$ & $p_{\infty}$ &  $p_2$ \\
\hline
$2.00\cdot 10^{-1}$ &  $6.20\cdot 10^{-5}$ & $7.40\cdot 10^{-5}$  & -  &  - \\
\hline 
$1.00\cdot 10^{-1}$ &  $1.09\cdot 10^{-5}$ & $1.43\cdot 10^{-5}$  & 2.51& 2.37  \\
\hline
$5.00\cdot 10^{-2}$ & $2.13\cdot 10^{-6}$ & $3.41\cdot 10^{-6}$  & 2.36  &  2.07\\
\hline 
$2.50\cdot 10^{-2}$ &  $5.42\cdot 10^{-7}$ & $1.00\cdot 10^{-6}$  & 1.97  & 1.77\\
\hline
 \end{tabular}
\caption{ Errors and convergence rates for the approximation of the value function of problem \eqref{test2} with $d=1$, $\sigma^2/2=0.05$, and $\Delta t = (\Delta x)^{4/3}/4$.}
\label{table:testsolesatta1nuovaV}
\end{table}

\begin{table}
\begin{tabular}{||c|c|c|c|c|c|c|c|c||} 
\hline
\multirow{2}{*}{$\Delta x$ }&\multicolumn{4}{c|}{High-order scheme} 
&\multicolumn{4}{c||}{Low-order scheme}\\
\cline{2-9}
&  $E_{\infty}$ &  $E_2$ & $p_{\infty}$ &  $p_2$&  $E_{\infty}$ &  $E_2$ & $p_{\infty}$ &  $p_2$ \\
\hline
$2.00\cdot 10^{-1}$ & $2.22\cdot 10^{-2}$ & $2.32\cdot 10^{-2}$  & -  &  -  & $1.90\cdot 10^{-1}$ & $1.84\cdot 10^{-1}$& -  &  - \\
\hline 
$1.00\cdot 10^{-1}$ &  $5.43\cdot 10^{-3}$ & $5.10\cdot 10^{-3}$  & 2.03 & 2.19 &$1.56\cdot 10^{-1}$ & $1.41\cdot 10^{-1}$&0.28 &0.38\\
\hline
$5.00\cdot 10^{-2}$ & $9.32\cdot 10^{-4}$ & $8.90\cdot 10^{-4}$  & 2.54 &  2.52 &$1.14\cdot 10^{-1}$ &$1.01\cdot 10^{-1}$&0.45&0.48\\
\hline 
$2.50\cdot 10^{-2}$ &  $1.33\cdot 10^{-4}$ & $1.26\cdot 10^{-4}$  & 2.81 &  2.82 
&$7.77\cdot 10^{-2}$&$6.86\cdot 10^{-2}$&0.55&0.56 \\
\hline
 \end{tabular}
\caption{Errors and convergence rates for the approximation of the density of problem \eqref{test2} with $d=1$, $\sigma^2/2=0.05$, and $\Delta t = (\Delta x)^{4/3}/4$.} 
\label{table:testsolesatta1nuova}
\end{table}
\begin{table}[h!]
  \centering
\begin{tabular}{||c|c|c|c|c||} 
\hline
\multirow{2}{*}{$\Delta x$ }&  \multicolumn{4}{c||}{Errors for the approximation of $v^*(0,\cdot)$}\\
\cline{2-5}
& $E_{\infty}$ &  $E_2$ & $p_{\infty}$ &  $p_2$ \\
\hline
$2.00\cdot 10^{-1}$ &  $1.68\cdot 10^{-4}$ & $1.70\cdot 10^{-4}$  & -  &  - \\
\hline 
$1.00\cdot 10^{-1}$ &  $3.56\cdot 10^{-5}$ & $3.48\cdot 10^{-5}$  & 2.24& 2.29\\
\hline
$5.00\cdot 10^{-2}$ & $5.86\cdot 10^{-6}$ & $5.75\cdot 10^{-6}$  & 2.60  &  2.60 \\
\hline 
$2.50\cdot 10^{-2}$ &  $1.06\cdot 10^{-6}$ & $1.04\cdot 10^{-6}$  & 2.47  & 2.47 \\
\hline 
 \end{tabular}
\caption{Errors and convergence rates for the approximation of  the value function of problem \eqref{test2} with $d=1$, $\sigma^2/2=0.005$, and $\Delta t = (\Delta x)^{4/3}/4$.}
\label{table:testsolesatta2nuovav}
\end{table}
\begin{table}[h!]
 \centering
\begin{tabular}{||c|c|c|c|c|c||} 
\hline
\multirow{2}{*}{$\Delta x$ }& \multicolumn{5}{c||}{Errors for the approximation of $m^*(T,\cdot)$}  \\
\cline{2-6}
& $E_{\infty}$ &  $E_2$ & $p_{\infty}$ &  $p_2$ & positivity error\\
\hline
$2.00\cdot 10^{-1}$ & $8.81\cdot 10^{-3}$ & $1.01\cdot 10^{-2}$  & -  &  - & $-3.51\cdot 10^{-4}$\\
\hline 
$1.00\cdot 10^{-1}$  & $3.06\cdot 10^{-3}$ & $2.53\cdot 10^{-3}$  & 1.53 & 2.00 & $-9.45 \cdot 10^{-9}$\\
\hline
$5.00\cdot 10^{-2}$  & $8.01\cdot 10^{-4}$ & $5.56\cdot 10^{-4}$  & 1.93 &  2.19 & $0$\\
\hline 
$2.50\cdot 10^{-2}$  & $1.81\cdot 10^{-4}$ & $1.14\cdot 10^{-4}$  & 2.15 &  2.29 & $0$\\
\hline
\end{tabular}
\caption{Errors and convergence rates for the approximation of  the density of problem \eqref{test2} with $d=1$, $\sigma^2/2=0.005$, and $\Delta t = (\Delta x)^{4/3}/4$.}
\label{table:testsolesatta2nuova}
\end{table}
\begin{figure}[h!]
\centering
     \begin{subfigure}[b]{0.45\textwidth}
        \centering
        \includegraphics[width=\textwidth]{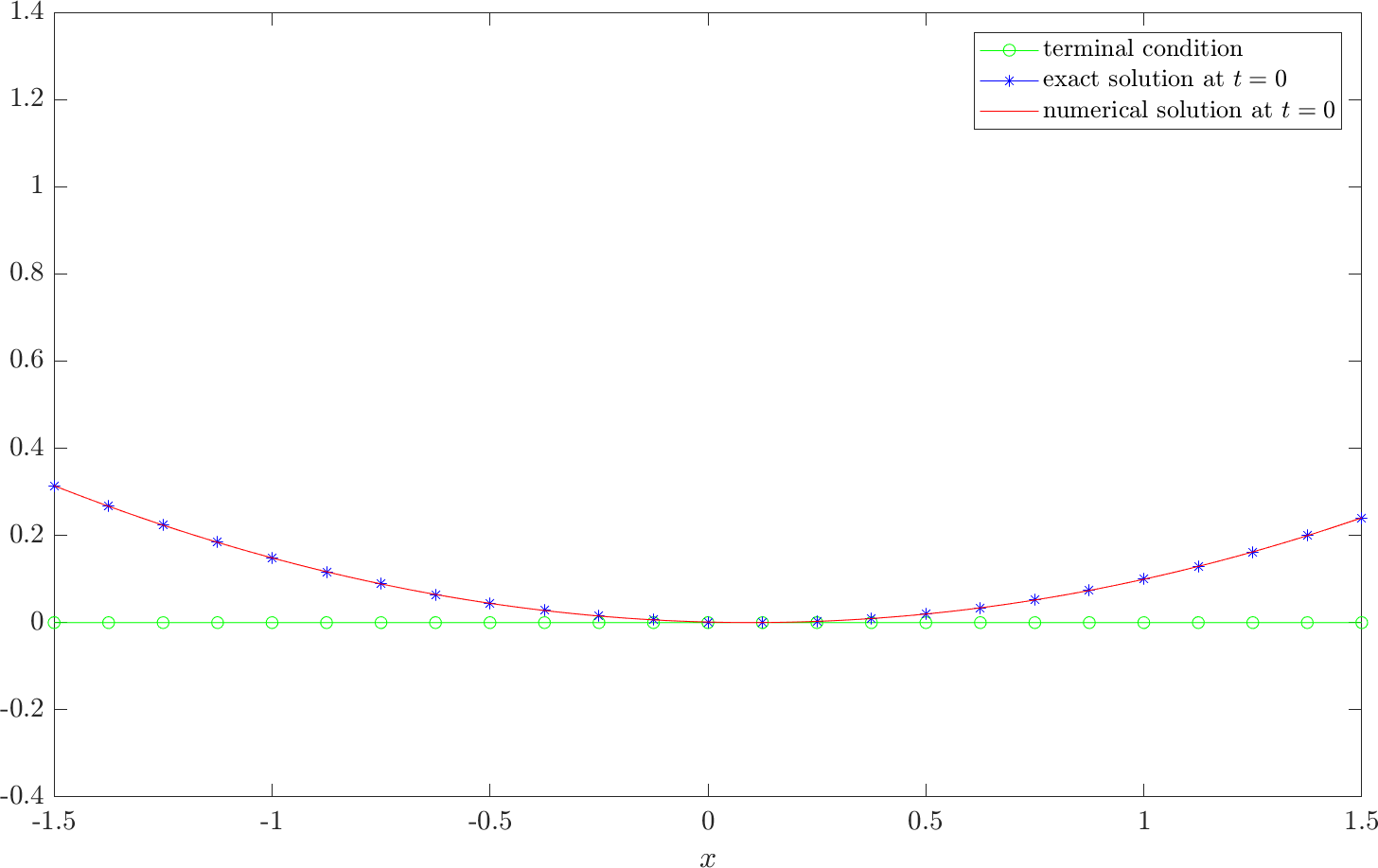}

     \end{subfigure} \quad \quad
     \begin{subfigure}[b]{0.45\textwidth}
         \centering
         \includegraphics[width=\textwidth]{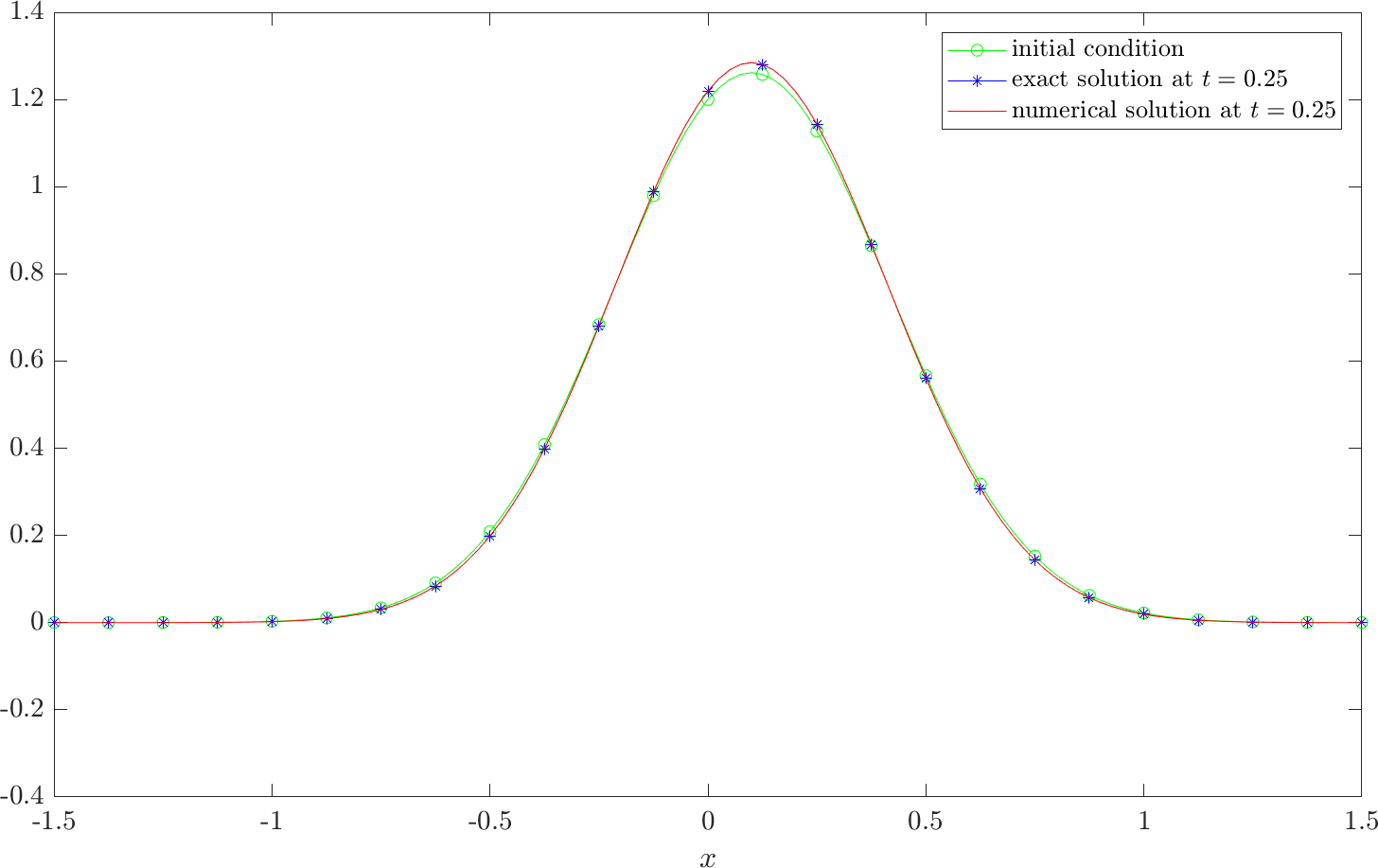}
     \end{subfigure}

	\caption{Solution to \eqref{test2} on $[0,T] \times \OO_\Delta = [0,0.25] \times (-2,2)$ with $\sigma^2/2=0.005$.  On the left, we display the exact value function $v^*$ at times $t=0$, $t=0.25$, and the numerical approximation $v_{\Delta}$ at time $t=0$. On the right, we display the exact density $m^*$ at times $t=0$, $t=0.25$, and   the numerical approximation $m_{\Delta}$ at time $t=0.25$.
	}\label{fig:testesattafig1}
\end{figure}
\begin{figure}[h!]
\includegraphics[width=0.4\textwidth]{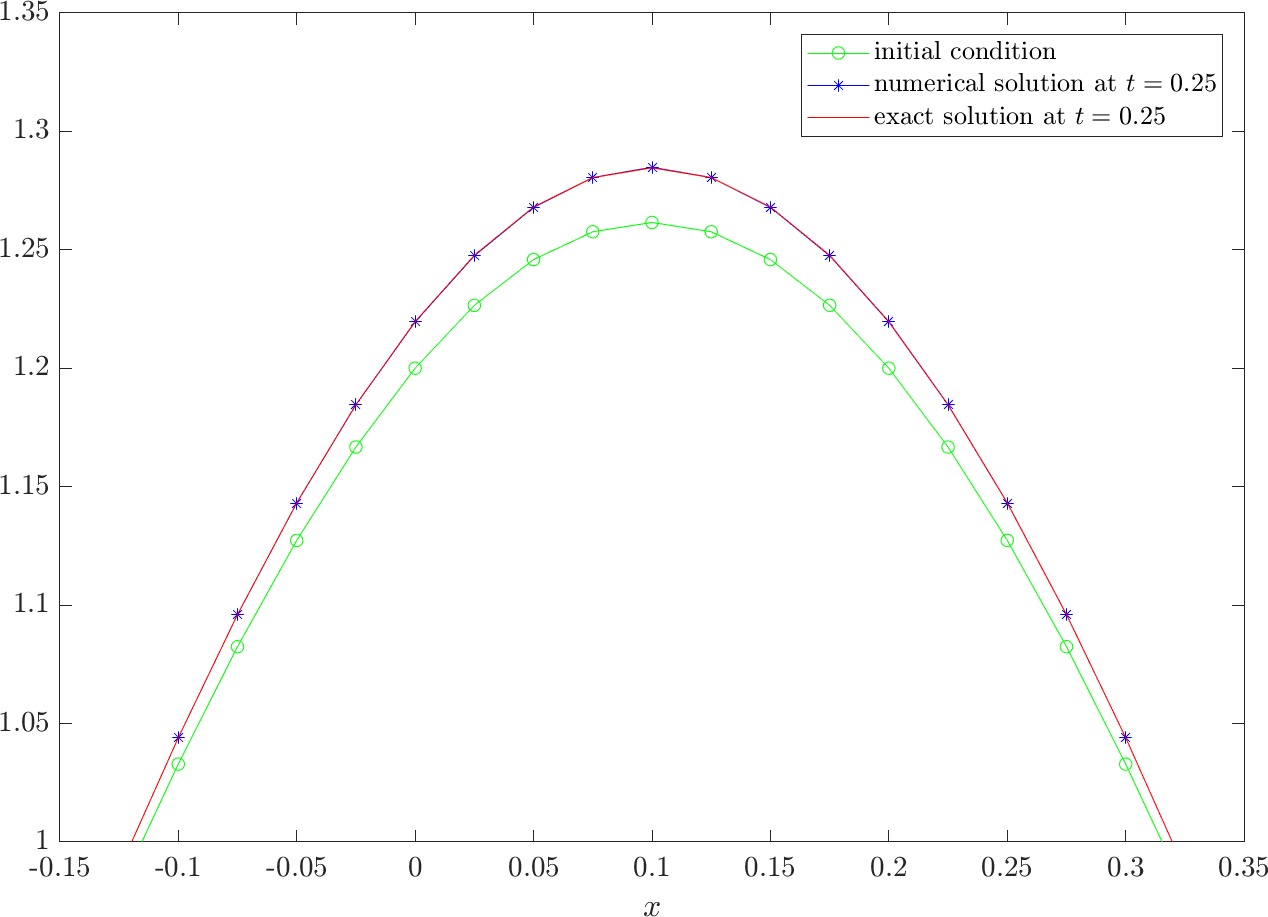} 
\caption{Zoom of the exact density $m^*$ at times $t=0$, $t=0.25$, and of the numerical approximation $m_{\Delta}$ at time $t=0.25$.}
\label{fig:testesatta2zoom}\end{figure}

\begin{table}[h!]
  \centering
\begin{tabular}{||c|c|c|c|c||} 
\hline
\multirow{2}{*}{$\Delta x$ }&  \multicolumn{4}{c||}{Errors for the approximation of $v^*(0,\cdot)$}\\
\cline{2-5}
& $E_{\infty}$ &  $E_2$ & $p_{\infty}$ &  $p_2$ \\
\hline
$2.00\cdot 10^{-1}$ &  $2.72\cdot 10^{-4}$ & $2.56\cdot 10^{-4}$  & -  &  - \\
\hline 
$1.00\cdot 10^{-1}$ &  $7.62\cdot 10^{-5}$ & $6.72\cdot 10^{-5}$  & 1.84& 1.93\\
\hline
$5.00\cdot 10^{-2}$ & $1.61\cdot 10^{-5}$ & $1.44\cdot 10^{-5}$  & 2.24 &  2.22 \\
\hline 
$2.50\cdot 10^{-2}$ &  $3.69\cdot 10^{-6}$ & $3.59\cdot 10^{-6}$  & 2.13  & 2.00 \\
\hline 
 \end{tabular}
\caption{Errors and convergence rates for the approximation of  the value function of problem \eqref{test2}  and $\sigma^2/2=0.005$, $\Delta t = \Delta x/4$.}
\label{table:testsolesattadxfratto4v}
\end{table}

\begin{table}[h!]
 \centering
\begin{tabular}{||c|c|c|c|c|c||} 
\hline
\multirow{2}{*}{$\Delta x$ }& \multicolumn{5}{c||}{Errors for the approximation of $m^*(T,\cdot)$}  \\
\cline{2-6}
& $E_{\infty}$ &  $E_2$ & $p_{\infty}$ &  $p_2$ & positivity error\\
\hline
$2.00\cdot 10^{-1}$ & $5.93\cdot 10^{-3}$ & $7.01\cdot 10^{-3}$  & -  &  - & $-8.18\cdot 10^{-5}$\\
\hline 
$1.00\cdot 10^{-1}$  & $2.63\cdot 10^{-3}$ & $2.17\cdot 10^{-3}$  & 1.17 & 1.69 & $-3.58 \cdot 10^{-10}$\\
\hline
$5.00\cdot 10^{-2}$  & $1.23\cdot 10^{-3}$ & $4.80\cdot 10^{-4}$  & 1.10 &  2.18 & $0$\\
\hline 
$2.50\cdot 10^{-2}$ & $3.39\cdot 10^{-4}$ & $9.61\cdot 10^{-5}$ & 1.86 & 2.32 & $0$\\
\hline
\end{tabular} 
\caption{Errors and convergence rates for the approximation of  the density of problem \eqref{test2} and $\sigma^2/2=0.005$, $\Delta t = \Delta x/4$.}
\label{table:testsolesattadxfratto4m}
\end{table}

\begin{table}[h!]
  \centering
\begin{tabular}{||c|c|c|c|c||} 
\hline
\multirow{2}{*}{$\Delta x$ }&  \multicolumn{4}{c||}{Errors for the approximation of $v^*(0,\cdot)$}\\
\cline{2-5}
& $E_{\infty}$ &  $E_2$ & $p_{\infty}$ &  $p_2$ \\
\hline
$2.00\cdot 10^{-1}$ &  $1.98\cdot 10^{-4}$ & $1.87\cdot 10^{-4}$  & -  &  - \\
\hline 
$1.00\cdot 10^{-1}$ &  $2.84\cdot 10^{-5}$ & $2.86\cdot 10^{-5}$  & 2.80 & 2.71\\
\hline
$5.00\cdot 10^{-2}$ & $3.41\cdot 10^{-6}$ & $3.94\cdot 10^{-5}$ & 3.06 & 2.86 \\
\hline 
$2.50\cdot 10^{-2}$ & $4.56\cdot 10^{-7}$ & $5.08\cdot 10^{-6}$ & 2.90 & 2.96 \\
\hline 
 \end{tabular}
\caption{Errors and convergence rates for the approximation of the value function of problem \eqref{test2}  and $\sigma^2/2=0.005$, $\Delta t = (\Delta x)^2$.}
\label{table:testsolesattadxquadrov}
\end{table}

\begin{table}[h!]
 \centering
\begin{tabular}{||c|c|c|c|c|c||} 
\hline
\multirow{2}{*}{$\Delta x$ }& \multicolumn{5}{c||}{Errors for the approximation of $m^*(T,\cdot)$} \\
\cline{2-6}
& $E_{\infty}$ & $E_2$ & $p_{\infty}$ & $p_2$ & positivity error\\
\hline
$2.00\cdot 10^{-1}$ & $6.60\cdot 10^{-3}$ & $7.63\cdot 10^{-3}$ & - & - & $-1.01\cdot 10^{-4}$\\
\hline 
$1.00\cdot 10^{-1}$ & $3.11\cdot 10^{-3}$ & $2.60\cdot 10^{-3}$ & 1.09 & 1.55 & $-1.19\cdot 10^{-8}$\\
\hline
$5.00\cdot 10^{-2}$ & $9.17\cdot 10^{-4}$ & $7.16\cdot 10^{-4}$ & 1.76 & 1.86 & $0$\\
\hline 
$2.50\cdot 10^{-2}$ & $2.42\cdot 10^{-4}$ & $1.81\cdot 10^{-4}$ & 1.92 & 1.98 & $0$\\
\hline
\end{tabular}
\caption{Errors and convergence rates for the approximation of the density of problem \eqref{test2}  and $\sigma^2/2=0.005$, $\Delta t = (\Delta x)^2$.}
\label{table:testsolesattadxquadrom}
\end{table}

\subsection{Mean field games with local couplings}\label{Popov}
In this section, we approximate the solution of the second-order MFG system with local couplings studied in \cite[Section 5.2]{PT15}. Namely, we consider system 
\be\label{test3}
\ba{rcl}
-\partial _t v - \frac{\sigma^2}{2}\Delta v + \half \left| \nabla v  \right|^2 &=&F(x,m)
\quad  \text{in } [0,T) \times \OO_\Delta, \\[10pt]
\partial_t m  - \frac{\sigma^2}{2} \Delta m  - \divergence\left(\nabla v m\right) &=&0 \quad  \text{in }(0,T] \times \OO_\Delta, \\[5pt]
v(T,\cdot) = 0, \; &\;& \;  m(0,\cdot) =m_0^*    \quad  \text{in } \RR^d,
\ea
\ee
with $T=0.05$, $\OO_\Delta =(0,1)$, homogeneous Neumann boundary conditions at $x=0$ and $x=1$, $\sigma^2/2=0.05$,
$$ m_0^*(x)=\begin{cases} 4 \sin^2(2\pi x-\frac{1}{4}) & \text{if }x\in[\frac{1}{4},\frac{3}{4}],\\
0&{\mbox{otherwise}},
\end{cases}\quad{\textrm{and}}\quad   F(x,m)=3m_0^{*}(x)-\min(4,m).$$

Notice that the coupling term $F$ depends on the density $m$ in a pointwise (or local) manner. The homogeneous Neumann boundary conditions are approximated as in \cite{CCDS21}. In this example, we do not have an explicit expression for $(v^*,m^*)$. 

We consider two cases, one with $\sigma^2/2 = 0.05$ and the other one with $\sigma^2/2 = 0.005$. 
 As input parameters for the fixed-point iterations, we set the initial guess for the density equal to the initial datum $m_0^*$ at each time step, as damping parameter $\omega=0.5$ and as tolerance
 $\tau=10^{-9}$.
In order to compute the errors and rates of convergence, we compare our approximations $(v_\Delta, m_\Delta)$ with a reference solution, which is still denoted by $(v^*,m^*)$, computed with $\Delta x= 6.67\cdot 10^{-4}$ and $\Delta t=(\Delta x)^{3/2}/3$.
In Tables \ref{tab:PopovHJ} and \ref{tab:Popovmass}, we show the errors and convergence rates for $v_{\Delta}(0,\cdot)$, $\partial_{x} v_{\Delta}(0,\cdot)$, and $m_{\Delta}(T,\cdot)$, which are computed by taking $\Delta t=(\Delta x)^{3/2}/3$ for different values of $\Delta x$.
We observe an order of convergence greater than two in most of the cases. In order to show the main advantage of the proposed scheme over a low-order scheme, we compare the proposed scheme with a first-order semi-Lagrangian scheme for MFG, developed in \cite{CS15}. Table \ref{tab:Popovmass}, columns 4 to 8, shows that, using the low-order scheme, the density is approximated with a much lower accuracy and the convergence rate is also much lower. Finally, Figure~\ref{fig:testesatta2} shows the approximated density $m_{\Delta}$ at time $t=T$ and the approximated value function $v_{\Delta}$, together with its gradient $D v_{\Delta}$,  at time $t=0$. These approximations are computed with $\Delta x = 1.56 \cdot 10^{-3}$.
\begin{table}[h!]
  \centering
\begin{tabular}{||c|c|c|c|c|c|c|c|c||} 
\hline
\multirow{2}{*}{$\Delta x$ }&  \multicolumn{4}{c|}{Errors for the approximation of $v^*(0,\cdot)$}& \multicolumn{4}{c||}{Errors for the approximation of $\partial_x v^*(0,\cdot)$}\\
\cline{2-9}
& $E_{\infty}$ &  $E_2$ & $p_{\infty}$ &  $p_2$ & $E_{\infty}$ &  $E_2$ & $p_{\infty}$ &  $p_2$  \\
\hline
$5.00 \cdot 10^{-2}$ & $5.38\cdot 10^{-2}$ & $3.80\cdot 10^{-2}$   & -  &  - &$8.09\cdot 10^{-2}$ & $4.96\cdot 10^{-2}$   & -  &  -  \\
\hline 
$2.50 \cdot 10^{-2}$  & $1.43\cdot 10^{-2}$ & $1.29\cdot 10^{-2}$  & 1.91 & 1.55&  $1.37\cdot 10^{-2}$ & $1.19\cdot 10^{-2}$  & 2.53& 2.05\\
\hline
$1.25 \cdot 10^{-2}$  &$4.25\cdot 10^{-3}$ & $3.24\cdot 10^{-3}$  & 1.75 &  1.99& $3.94\cdot 10^{-3}$ & $2.79\cdot 10^{-3}$  & 1.80  &  2.09 \\
\hline 
$6.25 \cdot 10^{-3}$ & $8.84\cdot 10^{-4}$ & $7.99\cdot 10^{-4}$  & 2.27 &  2.02 &  $8.34\cdot 10^{-4}$ & $7.07\cdot 10^{-4}$  & 2.23  & 1.98  \\
\hline
 \end{tabular}
\caption{\label{tab:PopovHJ} Errors and convergence rates for  the approximation of $v^*(0,\cdot)$ and $\partial_x v^*(0,\cdot)$ of problem \eqref{test3}, with $\Delta t = (\Delta x)^{3/2}/3$}
\end{table}
\begin{table}[h!]
  \centering
\begin{tabular}{||c|c|c|c|c|c|c|c|c||} 
\hline
\multirow{2}{*}{$\Delta x$ }&\multicolumn{4}{c|}{High-order scheme}  
&\multicolumn{4}{c||}{Low-order scheme}\\
\cline{2-9}
 & $E_{\infty}$ &  $E_2$ & $p_{\infty}$ &  $p_2$   & $E_{\infty}$ &  $E_2$ & $p_{\infty}$ &  $p_2$ \\
\hline
 $5.00 \cdot 10^{-2}$&$9.07\cdot 10^{-2}$ & $4.82\cdot 10^{-2}$   & -  &  - &$1.97\cdot 10^{-0}$ & $6.30\cdot 10^{-1}$   & -  &  -  \\
\hline 
$2.50 \cdot 10^{-2}$&  $1.81\cdot 10^{-2}$ & $6.79\cdot 10^{-3}$  & 2.32& 2.82  &$1.15\cdot 10^{-0}$ & $3.63\cdot 10^{-1}$   & 0.77 &  0.79 \\
\hline
$1.25 \cdot 10^{-2}$& $4.81\cdot 10^{-3}$ & $1.36\cdot 10^{-3}$  & 1.91  &  2.32  &$9.03\cdot 10^{-1}$ & $2.80\cdot 10^{-1}$   & 0.35 &  0.48 \\
\hline 
 $6.25 \cdot 10^{-3}$&  $7.64\cdot 10^{-4}$ & $2.06\cdot 10^{-4}$  & 2.65  & 2.72 &$1.96\cdot 10^{-1}$ & $4.96\cdot 10^{-2}$   & 0.48 &  0.51  \\
\hline 
 \end{tabular}
\caption{\label{tab:Popovmass}Errors and convergence rates for  the approximation of  $m^*(T,\cdot)$ of problem \eqref{test3}, with $\Delta t = (\Delta x)^{3/2}/3$}
\end{table}

\begin{figure}[h!]
 \includegraphics[width=0.3\textwidth]{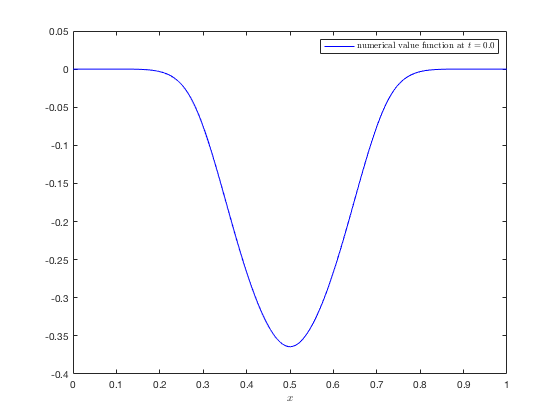} \includegraphics[width=0.3\textwidth]{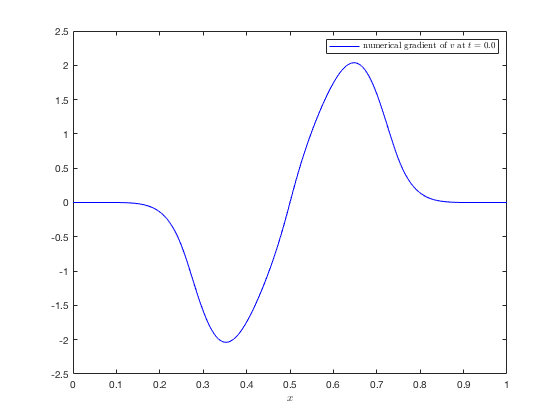} 
 \includegraphics[width=0.3\textwidth]{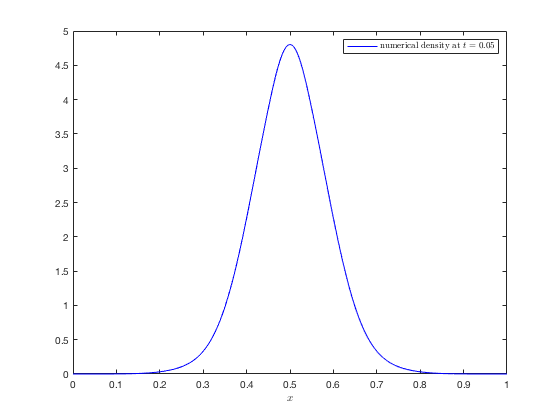}
\caption{Approximated value function $v_{\Delta}(0,\cdot)$ (left), the derivate $Dv_{\Delta}[\mu](0,\cdot)$ (center), and approximated density $m_{\Delta}(T,\cdot)$ of problem  \eqref{test3}(right).}
\label{fig:testesatta2}
\end{figure}

\section{Conclusions and future perspectives}
The main aim of the paper is to present a new and efficient high-order scheme to solve MFG systems with regular solutions. In order to do so,
we have developed a new high-order scheme for the \eqref{FP} equation, based on Lagrange-Galerkin methods combined with a second-order weak approximation of the underlying stochastic characteristic curves. A convergence analysis has been provided in the distributional sense and with respect to the weak topology in $L^2$. We have then combined the new scheme for the \eqref{FP} equation with a high-order semi-Lagrangian scheme for the \eqref{hjbmu} equation to obtain a high-order scheme for the \eqref{MFG} system.
We have shown the performance of the scheme  by numerical simulations.
We heuristically expect convergence rate $8/3$, which is reached in some cases.
The main advantage of the scheme, as usual for semi-Lagrangian schemes, is that during the fixed point iterations the equations \eqref{hjbmu} and \eqref{FP} are solved by schemes which are explicit and do not require the standard parabolic CFL condition ${\Delta t}=O((\Delta x)^2)$ in order to be stable. Recall that the CFL condition is required by standard explicit finite difference schemes to approximate parabolic PDEs. Instead, in  Theorem \ref{weakconvergence} and in Proposition \ref{propiedadesbasicasesquema}    the relation  ${(\Delta x})^{q+1}/{\Delta t}\to 0$ is assumed, which implies that larger time steps than $\Delta t=O((\Delta x)^2)$ are allowed. The restriction on the  time step is only due to accuracy. In fact, as Tables \ref{table:testsolesattadxfratto4m} and
\ref{table:testsolesattadxquadrom} show, the accuracy may decrease for time steps that are far from the optimal one. Similar considerations were observed in \cite{ferrettisecondoordine}, where high-order semi-Lagrangian schemes are applied to approximate linear parabolic PDEs.
In addition, the scheme for the \eqref{FP} is conservative, which is not generally true for semi-Lagrangian type schemes applied to conservation laws (see e.g. \cite{BF14}).
The main drawbacks of the scheme are the loss of positivity for the discrete density, and the lack of a constant high-order convergence rate.
Both drawbacks are due to the choice of  standard cubic basis functions. Investigation on the use of a different class of basis function is an interesting point to be addressed in the future.
Concerning the observed oscillations  in the rate of convergence, we attribute them mostly to the lack of regularity of the integrands appearing in the schemes which yields, possibly, a lack of accuracy in the approximation of the integral terms. More regular basis functions may help to improve the quadrature error and then the overall truncation errors.\\

{\bf Acknowledgements.} The first two authors would like to thank the Italian Ministry of Instruction, University and Research (MIUR) for supporting this research with funds coming from the PRIN Project $2017$ ($2017$KKJP$4$X entitled ``Innovative numerical methods for evolutionary partial differential equations and applications''). The three authors were partially supported by KAUST through the subaward agreement OSR-$2017$-CRG$6$-$3452$.$04$.

\normalsize
\section{Appendix: Proofs.}
{{\em Proof of Theorem \ref{teorema_well_posedness} }
We refer the reader to \cite[Theorem~6.6.1]{MR3443169} for the existence result in {\rm(i)} as well as for the nonnegativity property in {\rm(ii)}. The uniqueness result in {\rm(i)} and the mass conservation property in {\rm(iii)} follow from \cite[Theorem~9.3.6]{MR3443169} and \cite[Corollary~6.6.6]{MR3443169}, respectively. Finally, the proof of {\rm(iv)} is given in \cite[Theorem 4.3]{figalli08}.
 {\hfill{$\square$}\medskip}

{\em{Proof of Proposition \ref {thm:equicontcons}}}. In the proof  of both assertions, we fix $\phi \in  C_{0}^{2}(\RR^d)$  and we will  denote by $C$  a positive real number which  can depend on $\phi$ but not on $\Delta$. We will also use the estimate 
\be\label{eq:stimaCNoperatore}  \left| \sum_{\ell \in \I_{d}}\omega_{\ell}\phi(y^\ell_k(x))  -\left[\phi(x)+\Dt \left( \frac{\sigma^2}{2} \Delta \phi(x)+  \langle b(x,t_k),\nabla \phi(x)\rangle \right)\right]\right| \leq C  (\Delta t)^{2} \quad \text{for $x\in \RR^d$,}\ee
which follows from the definition of $y^\ell_k(x)$ and a Taylor expansion (see for instance  \cite{BCCF21}).  \smallskip

{\rm(i)} Let us first show the assertion for $t=t_{k+1}$ and $s=t_{k}$ for some $k\in \I^*_{\Dt}$. Set  $\eps  : = \phi-I[\phi]$ and fix $k\in \I_{\Delta t}^{*}$.  Remark~\ref{uniform_integrability_over_compact_sets} yields the existence of $C>0$ such that
\be\label{lemma_3_2_primo_paso}\ba{rcl}
\ds \left| \int_{\RR^d} \phi (x) \left(m_{\Delta} (t_{k+1},x) - m_{\Delta} (t_{k},x) \right) \dd x\right| & \ds \leq &\ds \left| \int_{\RR^d} I[\phi] (x) \left(m_{\Delta} (t_{k+1},x) - m_{\Delta} (t_{k},x) \right) \dd x\right| \\[12pt]
\; & \, & + C  \|\varepsilon\|_{L^\infty}. 
\ea
\ee
Recalling that $\mbox{supp}\{m_{\Delta}(t_k,\cdot)\}\subset \OO_{\Delta}$ and using the definition of the scheme in \eqref{LG}, we have that
\be\label{lemma_3_2_secondo_paso}\ba{ll} \ds
 \int_{\RR^d} I[\phi] (x) \left(m_{\Delta} (t_{k+1},x) - m_{\Delta} (t_{k},x) \right) \dd x&=  \ds \int_{\OO_{\Delta}} \sum_{i\in \ZZ^d} \phi(x_i) \beta_i (x) \left(\sum_{j\in \I_{\Delta x}} \left(m_{k+1,j}  -m_{k,j}\right)  \beta_j(x) \right) \dd x \\[6pt]
&= \ds \sum_{i\in \ZZ^d} \phi(x_i)  \left(\sum_{j\in \I_{\Delta x}} \left(m_{k+1,j}  -m_{k,j}\right) \int_{\OO_{\Delta}}\beta_i (x) \beta_j(x)\dd x  \right)\\[25pt]
& =\ds \sum_{i\in  \ZZ^d} \phi(x_i) \bigg[\sum_{\ell \in \I_d} \omega_\ell \sum_{j\in \I_{\Delta x}} m_{k,j} \bigg(\int_{\OO_{\Delta}} \beta_i (y^\ell_k(x)) \beta_j(x) \dd x \\[6pt]
& \hspace{5.2cm}\ds- \int_{\OO_{\Delta}} \beta_i (x) \beta_j(x) \dd x \bigg) \bigg]  
\\[12pt]
&= \ds  \sum_{\ell \in \I_d} \omega_\ell \sum_{j\in \I_{\Delta x}} m_{k,j}\int_{\OO_{\Delta}} \bigg[I[\phi] (y^\ell_k(x))  -  I[\phi] (x)\bigg]\beta_j(x) \dd x
\\[18pt]
&= \ds  \sum_{\ell \in \I_d} \omega_\ell \int_{\OO_{\Delta}} \bigg[I[\phi] (y^\ell_k(x))  -  I[\phi] (x)\bigg] m_{\Delta}(t_k,x)  \dd x.
\ea
\ee
On the other hand, since $\phi$ has a compact support, there exists $ C>0$ such that 
\be\label{non_lo_so}
\left\| \sum_{\ell \in \I_d} \omega_\ell \left(I[\phi] (y^\ell_k(\cdot)) -\phi(y^\ell_k(\cdot)) \right)\right\|_{L^2}     + 
\| \phi -  I[\phi]\|_{L^2 }   \leq  C \|\eps\|_{L^\infty}
\ee
and, by \eqref{eq:stimaCNoperatore} and {\bf(H2)},  there exists $C>0$ such that 
 \be\label{eq:stimaCNoperatore_semplice}
 \left\|  \sum_{\ell \in \I_d} \omega_\ell \left(\phi(y^\ell_k(\cdot)) -\phi \right)\right\|_{L^2 }  \leq  C \Delta t.
\ee
Thus, by the triangular and the Cauchy-Schwarz inequalities, Theorem~\ref{Prop:stability}{\rm(iii)}, \eqref{lemma_3_2_primo_paso}, \eqref{lemma_3_2_secondo_paso}, \eqref{non_lo_so},  and \eqref{eq:stimaCNoperatore_semplice}, we get  the existence of $C>0$ such that 
$$
\left| \int_{\RR^d} \phi (x) \left(m_{\Delta} (t_{k+1},x) - m_{\Delta} (t_{k},x) \right) \dd x\right| \leq C\left( \|\eps\|_{L^\infty }  +  \Delta t\right).
$$
It follows from \eqref{approximation_error_V_Deltax_q}, and the condition  $(\Delta x)^{q+1}\leq \Delta t$,  the existence of $C>0$ such that  \eqref{eq:prop_equicontinuity} holds for $t=t_{k+1}$ and $s=t_k$. Using this relation and the triangular inequality, we deduce that  \eqref{eq:prop_equicontinuity} holds for every $s=t_{k}$ and $t=t_{m}$ with $k$, $m\in \I_{\Dt}$.  

Now, let us fix $s,t\in [0,T]$ and assume, without loss of generality, that $t>s$. Let $k_1,\,  k_2\in \I_{\Delta t}^*$ be such that  $s \in [t_{k_1}, t_{k_1+1}]$ and $t \in [t_{k_2}, t_{k_2+1}]$. If $k_1=k_2$, then it follows  from \eqref{m_extension} that \eqref{eq:prop_equicontinuity} holds with $C_\phi=C$. Otherwise, $k_{2}\geq k_1+1$ and \eqref{m_extension} yield
\be
\label{proof:ineq_aux_1}
\ba{l}
\ds \left| \int_{\RR^d} \phi (x) \left(m_{\Delta} (t_{k_1+1},x) - m_{\Delta} (s,x) \right)  \dd x\right|  \leq \frac{t_{k_{1}+1}- s}{\Delta t} \left| \int_{\RR^d} \phi (x) \left(m_{\Delta} (t_{k_1+1},x) - m_{\Delta} (t_{k_1},x) \right)  \dd x \right|  \\[14pt]
\hspace{6.15cm} \leq  C(t_{k_{1}+1}- s).
\ea
\ee
Similarly, 
\be
\label{proof:ineq_aux_2}
\left| \int_{\RR^d} \phi (x) \left(m_{\Delta} (t_{k_2},x) - m_{\Delta} (t,x) \right)  \dd x \right|  \leq C( t- t_{k_2}).
\ee
Altogether, it follows from the triangular inequality,   \eqref{proof:ineq_aux_1}, \eqref{proof:ineq_aux_2}, and \eqref{eq:prop_equicontinuity}, with $t=t_{k_2}$ and $s= t_{k_{1}+1}$, that \eqref{eq:prop_equicontinuity} holds with $C_{\phi}=C$. \vspace{0.2cm}

{\rm(ii)}  By \eqref{approximation_error_V_Deltax_q},  Remark~\ref{uniform_integrability_over_compact_sets}, and the definition of the scheme \eqref{LG}, for each $k \in \I^*_{\Delta t}$  we have 
\be\label{serieecsdemfinal}
\ba{rcl}
\ds \int_{\RR^d}\phi(x) m_{\Delta}(t_{k+1},x) \dd x &=& \ds \int_{\RR^d}I[\phi](x) m_{\Delta}(t_{k+1},x) \dd x+ O(({\Delta x})^{q+1})\\[12pt]
	\; &=&  \ds \sum_{i\in \ZZ^d} \phi(x_i) \sum_{j\in \I_{\Delta x_n}}   m_{ k+1,j}\int_{\RR^d} \beta_i(x)\beta_j(x) \dd x  + O(({\Delta x})^{q+1})\\[12pt]
	&=& \ds \sum_{i\in \ZZ^d} \phi(x_i) \sum_{j\in \I_{\Delta x}}   m_{k,j}\sum_{\ell \in \I_{d}}\omega_{\ell}\int_{\RR^d} \beta_i(y^\ell_k(x))\beta_j(x) \dd x  +O(({\Delta x})^{q+1})	\\[15pt]
	&=& \ds   \sum_{j\in \I_{\Delta x_n}}   m_{k,j}\sum_{\ell \in \I_{d}}\omega_{\ell}\int_{\RR^d} I[\phi](y^\ell_k(x))\beta_j(x) \dd x  +O(({\Delta x})^{q+1})	\\[12pt]
	&=& \ds    \int_{\RR^d}\left(\sum_{\ell \in \I_{d}}\omega_{\ell} \phi(y^\ell_k(x))\right)m_{\Delta}(t_{k},x) \dd x +O(({\Delta x})^{q+1}).
\ea\ee
Using \eqref{eq:stimaCNoperatore} and Remark~\ref{uniform_integrability_over_compact_sets}, we obtain
\be
\label{eq_consistency_with_errors}
\ba{rcl}
		\ds \int_{\RR^d}\phi(x) \left(m_{\Delta}(t_{k+1},x) \right.-\left. m_{\Delta}(t_{k},x)\right)\dd x 
		&=& \ds \Dt\int_{\RR^d}\left( \frac{\sigma^2}{2} \Delta \phi(x)+ \langle b(t_k,x), \nabla \phi(x)\rangle\right)m_{\Delta}(t_{k},x) \dd x  \\[12pt]
		\;&\; &\ds  +O(({\Delta x})^{q+1}+({\Delta t})^{2}).\ea
\ee
Notice that, for any $s\in [t_{k}, t_{k+1}]$, Remark~\ref{uniform_integrability_over_compact_sets} implies that
\be
\label{eq_consistency_error_integral_of_b}
\int_{\RR^d}  \langle b(t_k,x), \nabla \phi(x)\rangle m_{\Delta}(t_{k},x) \dd x= \int_{\RR^d}  \langle b(s,x), \nabla \phi(x)\rangle m_{\Delta}(t_{k},x) \dd x + O(\omega_\phi(\Delta t)). 
\ee
By \eqref{m_extension} and the fact that $b(s,\cdot)\in C^{q+1}(\RR^d)$,  together with assertion {\rm(i)},   we have
\be\label{eq:stima1}
\left|\int_{t_k}^{t_{k+1}}\int_{\RR^d}\left(\frac{\sigma^2}{2} \Delta \phi(x)+ \langle b(s,x), \nabla \phi(x)\rangle\right)(m_{\Delta}(s,x)-m_{\Delta}(t_k,x))\dd x\dd s\right|= O((\Delta t)^2).
\ee
Thus,  \eqref{prop:consistency_eq} follows from \eqref{eq_consistency_error_integral_of_b}, \eqref{eq:stima1}, and \eqref{eq_consistency_with_errors}.
 {\hfill{$\square$}\medskip}

{\em{Proof of Lemma \ref{compatezza_in_D_primo}.}} In view of the Arzel\`a-Ascoli theorem \cite[Chapter 7, Theorem 18]{MR0070144} (see also \cite[Section 4]{Krukowski18}) and Proposition~\ref{thm:equicontcons}{\rm(i)}, it suffices to show that the family $\M$ is pointwise relatively compact. Let us consider the absolutely convex set $U_0:=\{\phi \in C^{\infty}_0(\RR^d)\; |\; \|\phi\|_{L^{\infty} }<1, \; \mbox{supp} \,\phi \subseteq \ov B(0,1)\}$. This set is a neighborhood of $0$ in the standard topology of $C_{0}^{\infty}(\RR^d)$ (see e.g. \cite[Chapter 10]{MR4182424}) and, for any $t\in [0,T]$,  
$$
\sup_{\phi \in U_0}\left| \int_{\RR^d} m_{\Delta}(t,x) \phi(x)\dd x\right| = \sup_{\phi \in U_0}\left| \int_{\ov B(0,1)} m_{\Delta}(t,x) \phi(x)\dd x\right|\leq \|  m_{\Delta}(t,\cdot)\|_{L^1({\ov B(0,1)})}\leq r,
$$
where  $r:= \sup  \{\|  m_{\Delta}(t,\cdot)\|_{L^1({\ov B(0,1)})} \, | \, \Delta \in (0,\infty)^2\}$ belongs to $[0,+\infty)$ by \eqref{integrale_1_stima_l2}. This proves that $\{m_{\Delta }(t, \cdot) \, | \, \Delta \in (0,\infty)^2\} \subset \left\{ T \in \D'(\RR^d) \,  | \, \sup_{\phi \in U_0} \left|T(\phi)\right | \leq r \right\}$ which, by the Banach-Alaoglu-Bourbaki theorem (see e.g. \cite[Theorem 23.5]{MR1483073}), is a compact subset of $\D'(\RR^d)$. 
 {\hfill{$\square$}\medskip}

\em{Proof of Proposition \ref{propiedadesbasicasesquema}.}} Let $\Delta t>0$, $\Delta x>0$, and $\alpha\in A$. In the computations below, the big $O$ terms are uniform with respect to $\alpha\in A$.    Let us apply  \eqref{eq:stimaCNoperatore} to $\phi( t_{k+1},\cdot)$, with $b(t,x)=-\alpha$, to obtain 
\be\label{eq:eqlemma1}\ba{rcl}
\underset{\ell \in \I_{d}}\sum \omega_\ell \phi \left( t_{k+1},x_i -\Delta t \alpha+\sqrt{\Delta t}\sigma e^\ell \right) &=& \phi \left(t_{k+1}, x_i \right) +  \Dt \left(  \frac{\sigma^2 }{2} \Delta  \phi(t_{k+1},x_i)- \langle \nabla \phi(t_{k+1},x_i),\alpha\rangle \right)\\
& &+  O \left( (\Dt)^2  \right).
\ea
\ee
By ({\bf{H2}}) and using the first-order Taylor expansion of $F(\cdot , \mu(t_{k}))$ around $x_i$,  we get
\be\label{eq:eqlemma2}
\ba{rcl}
\frac{1}{2}\left(\underset{\ell \in \I_{d}}\sum \omega_\ell F(x_i -\Delta t \alpha+\sqrt{\Delta t}\sigma e^\ell, \mu(t_{k+1})) +F(x_i,\mu(t_{k}))\right)&=&F( x_i,\mu (t_{k+1})) \\[8pt]
\; & \, & + O (\Dt+{\bf d}(\mu(t_{k+1}),\mu(t_{k}))).
\ea
\ee
Thus, by \eqref{SLscheme}, \eqref{eq:eqlemma1}, \eqref{eq:eqlemma2}, and \eqref{approximation_error_V_Deltax_q},  we obtain 
$$
\ba{rcl}
S[\mu](\phi_{k+1},k,i) &=& \ds \phi \left(t_{k+1}, x_i \right) -\Dt \sup_{\alpha \in A} 
\left[\langle \nabla \phi(t_{k+1},x_i),\alpha\rangle - \frac{|\alpha |^{2}}{2} \right] 
+ \Dt \frac{\sigma^2 }{2}\Delta \phi(t_{k+1},x_i) \\[12pt]
\; & \; & \ds  + \Dt F(x_i, \mu(t_{k+1})) + O \left((\Dt)^2+ (\Dx)^{q+1}+\Delta t{\bf d}(\mu(t_{k+1}),\mu(t_{k}))\right) \\[6pt]
\; & = & \ds \phi \left( t_{k+1},x_i \right)- \frac{\Dt}{2} | \nabla \phi(t_{k+1},x_i)|^{2}+ \Dt \frac{\sigma^2 }{2} \Delta \phi(t_{k+1},x_i)\\[12pt]
& \; & + \Dt F(x_i, \mu(t_{k+1}))+ O \left( (\Dt)^2 + (\Dx)^{q+1}+ \Delta t{\bf d}(\mu(t_{k+1}),\mu(t_{k}))\right).
\ea
$$
Finally, we get 
\[
\begin{split}
 \frac{1}{{\Delta t}} \left[\phi(t_{k},x_i)-S_{\Delta}[\mu](\phi_{k+1},k,i)\right] =& -\partial _t\phi(t_{k+1},x_i) - \frac{\sigma^2 }{2} \Delta\phi(t_{k+1},x_i)+\frac{1}{2}| \nabla \phi(t_{k+1},x_i)|^{2} -  F(x_i, \mu(t_{k+1}))  \\
& + O \left(\Dt + \frac{(\Dx)^{q+1} }{\Dt} +{\bf d}(\mu(t_{k+1}),\mu(t_{k}))\right),
\end{split}
\]
from which the result follows.
 {\hfill{$\square$}\medskip}
\end{document}